\documentclass[reqno, a4paper, draft]{amsart}
\usepackage{amssymb, amscd, amsfonts,  amsthm}
\usepackage[mathscr]{eucal}
\usepackage{graphicx}
\usepackage[all,poly]{xy} 
 \usepackage[all]{xy} 

\newtheorem{theorem}{Theorem}[section]
\newtheorem{lemma}[theorem]{Lemma}
\newtheorem{corollary}[theorem]{Corollary}
\newtheorem{proposition}[theorem]{Proposition}

\newtheorem{fact}[theorem]{Fact}

\theoremstyle{definition}
\newtheorem{definition}[theorem]{Definition}

\newtheorem{remarks}[theorem]{Remarks}
\newtheorem{example}[theorem]{Example}

\DeclareMathOperator{\possf}{\mathsf{pos}}
\DeclareMathOperator{\negsf}{\mathsf{neg}}

\newcommand{\restrict}{\,{\mathbin{\vert\mkern-0.3mu\grave{}}}\,}

\newcommand{\remove}[1]{}

\DeclareMathOperator{\pow}{\rm pow}

\DeclareMathOperator{\interval}{[0,1]}

\DeclareMathOperator{\conv}{\rm conv}

\DeclareMathOperator{\cl}{\rm cl}

\DeclareMathOperator{\at}{\rm at}

\newcommand{\yesno}{\{0,1\}}

 \title[De Finetti for undergraduates]
{De Finetti for mathematics undergraduates} 

\author{Daniele Mundici}
 
\address[D. Mundici]{Department of
Mathematics and Computer Science  ``Ulisse Dini'' \\
University of Florence\\
Viale Morgagni 67/A \\
I-50134 Florence \\
Italy}
\email{daniele.mundici@unifi.it }

\date{}



\begin{document}

\begin{abstract} In 1931  
 de Finetti proved  
 what is known as his  Dutch Book Theorem.
This  result  implies
  that the  finite additivity {\it axiom} 
for the probability
of the disjunction of two incompatible events 
becomes a {\it consequence}  of de Finetti's  logic-operational 
consistency notion. 
Working in the context of boolean algebras we prove
 de Finetti's 
theorem. 
The mathematical background required is little more than that which is taught in high school.
As a preliminary step we prove 
  what de Finetti called ``the Fundamental Theorem of Probability'',
his  main contribution both to  Boole's probabilistic inference problem 
and to its modern 
reformulation known as  the optimization version of the 
probabilistic satisfiability  problem.
In a final section 
 we give a self-contained combinatorial proof of
de Finetti's exchangeability theorem.
 \end{abstract}

\maketitle
\section*{Introduction}
In his 1931 paper \cite{def-fundmath}, 
 de Finetti introduced his celebrated consistency
 notion  (also known as ``coherence''), and proved  
 what today  is known as his  Dutch Book Theorem.
In \cite[$\S$ 16, page 328]{def-fundmath}
he summarized his results as follows
(italics by de Finetti):

\medskip
\begin{quote} 
{\small Dimostrate le propriet\'a fondamentali del 
calcolo classico delle probabilit\'a, ne scende che
tutti i risultati di tale calcolo non sono che
{\it conseguenze} della definizione che abbiamo data
della {\it coerenza}.

\medskip
\noindent
(Having proved the fundamental properties
of the classical probability calculus, it follows that
all its results are nothing else but
{\it consequences} of the definition of
{\it consistency} \/ given in this paper.)
}
 \end{quote}

\noindent
Indeed, a main consequence of his Dutch Book theorem
is  that the finite additivity {\it  axiom} 
for the probability
of the disjunction of two incompatible events becomes a 
{\it consequence} of  de Finetti's  
consistency\footnote{The  Italian adjective  ``coerente'' (resp., the French
``coh\'erent'')
adopted by  de Finetti in his 
original paper  \cite{def-fundmath}
(resp., in his paper \cite{def-poincare}),
is translated  ``consistent'' in the present paper---for
the sake of terminological consistency.
As a matter of fact, 
 ``logical consistency''
is translated ``coerenza logica'' in Italian.
Moreover,
Corollary \ref{corollary:justification}  shows
that de Finetti's (probabilistic) consistency
is a generalization of logical consistency.} 
notion. 
Working in the context of boolean algebras, we offer a
 self-contained   proof of de Finetti's 
theorem. As a preliminary step we prove 
  what de Finetti called ``the Fundamental Theorem of Probability'',
his  main contribution both to  Boole's probabilistic inference problem 
and to its modern 
reformulation known as  the optimization version of the 
probabilistic satisfiability  problem.

In the final part of this paper we will give a self-contained
proof of de Finetti's exchangeability theorem.

\section{De Finetti's  ``Fundamental Theorem of Probability''}
\label{chapter:fundamental}

\pagenumbering{arabic}

In his paper \cite{def-fundmath} de Finetti writes (in his own
italics):

\begin{quote}
\small {Un evento $E$  \`e una proposizione, un'affermazione, che non sappiamo ancora se sia vera o falsa

\medskip
An event $E$ is a proposition, a statement, which we do not yet know whether it is true or false

\hfill  De Finetti,  \cite[\S 7 p. 307]{def-fundmath}
}
\end{quote}

\index{de Finetti!event}

\medskip
\noindent
\begin{quote}
\small {
Un individuo \`e coerente nel valutare le probabilit\`a di certi eventi se qualunque gruppo di puntate $S_1, S_2,\dots , S_n$ un competitore faccia su un insieme qualunque di eventi $E_1,E_2,\dots, E_n$ fra quelli che egli ha considerato, non \`e possibile che il guadagno $G$ del competitore risulti {\it in ogni caso positivo}. 

\medskip
An individual is consistent in 
evaluating the probabilities of certain events 
if for any set of stakes  
$ S_1, S_2, \dots, S_n $ a competitor places
on any set of events 
$ E_1, E_2, \dots, E_n $ 
among those he has considered, it is not possible for the competitor's  $G$ gain to be {\it  positive in any case}.

\medskip
\hfill  De Finetti,  \cite[\S 7 p. 308]{def-fundmath}
}
\end{quote}

 \index{event}
\index{possible world}
\index{de Finetti!consistency notion}
\index{de Finetti!event}

\noindent
This expository style in presenting 
 two of the most basic notions of 
de Finetti's theory  differs from the 
style adopted in this paper. 
The next few pages will be devoted to the 
{\it definition} of events and their outcomes in the context  of
boolean algebras, the {\it motivation} of these definitions, and 
the {\it representation} of events  
in euclidean finite-dimensional space. 
In Section \ref{section:recap} we will 
give concrete examples of events and their outcomes
in our algebraic framework.

%
%

\subsection{Boolean algebras and their homomorphisms}
\label{section:boolean}

While groups  are a mathematical counterpart 
to  ``symmetries'', boolean algebras 
provide a rigorous approach to the imprecise notion
of  ``event''.  Furthermore, 
the homomorphisms of  boolean algebras into
the two-element boolean algebra
$\yesno$ provide a convenient formal counterpart  
to  the   ``possible outcomes'' 
of these events.

\index{logic}
\index{proposition}
\index{boolean!logic}

Boolean algebras also provide an algebraic
counterpart to ``propositions'' in boolean logic, equipped with
the connectives of
negation, conjunction and disjunction.  Although  logic is
not the subject of this paper, the reader will have
various opportunities to see the
 mutual relationships between these two interpretations
of boolean algebras.

%

\bigskip
\noindent
An {\it algebra} is a nonempty set equipped with distinguished
constants and operations.  

\medskip
Following standard practice,
the mathematical neologism ``iff'' stands for ``if and only if''.
Thus for example, an even number is prime iff it equals 2.
Sometimes  we write  
$ \Leftrightarrow$ instead of ``iff''. 
\index{iff, if and only if}
\index{$\Leftrightarrow$, if and only if}

\index{$\wedge$, the infimum or  meet operation} 
\index{$\vee$, the supremum or join operation}

\index{lattice}
\index{least upper bound}
\index{join}
\index{greatest lower bound}
\index{meet}
\index{distributive property}
\index{supremum}
\index{infimum}

\begin{definition}
{\rm A {\em lattice} 
is an algebra $L=(L,\wedge,\vee)$
equipped with a partial 
order such that  any
two elements  $x,y\in L$ have 
a greatest lower bound 
 (alias the {\em infimum}, or {\em meet})  $x\wedge y$
 and a least upper bound
(also known as the {\em supremum}, or {\em join}) $x\vee y$. 
We say that $L$ is {\em distributive}
 if for all $x,y,z\in L$,\,\,\, 
$$
\mbox{$x\wedge(y\vee z)=(x\wedge y)\vee(x\wedge z)$ \,\,\,and 
\,\,\,$x\vee(y\wedge z)=(x\vee y)\wedge(x\vee z)$.}
$$
The {\em underlying order}
 of $L$ is defined by the stipulation
$x\leq y$ iff  $x \wedge y=x.$

 \smallskip
A {\em boolean algebra} 
$A=(A,0,1,\neg, \wedge,\vee)$
  is a distributive lattice
$(A, \wedge,\vee)$ with a smallest element $0$
and a largest element $1$,  equipped with an operation $\neg$
(called {\em complementation})  
 satisfying  $x\wedge \neg x=0$
and $x\vee \neg x=1.$\footnote{All boolean
algebras in this paper will satisfy the nontriviality condition $0\not=1.$}  

An element \,$b\in A$ is said to be 
an {\em atom} 
if it is a nonzero minimal element, in symbols,
 $$b\in \at(A).$$
 }
 \end{definition}
 
 \index{boolean!algebra}
\index{atom}
\index{complementation}
\index{$\neg$, complementation, negation}
\index{underlying order of a boolean algebra}

 \index{$\at(A)$, the set of atoms of $A$}
 
 \begin{example}
 The two-element boolean algebra
$\yesno=(\yesno,\neg, \wedge, \vee)$  is defined
by  $\neg 0=1, \,\,\,\neg 1=0, \,\,\,
a\wedge b=\min(a,b),\,\,\,  a\vee b=\max(a,b)$ for all
$a,b\in \yesno.$

 \end{example}
 
  ``Events'' and their mutual relations
 will be understood as
elements of a boolean algebra
$A$  acted upon by the operations of $A$.

 \index{homomorphism}
 \index{isomorphism} 
 
\begin{definition} 
\label{definition:homomorphism}
{\rm  A {\em homomorphism}
  of a boolean algebra $A=
 (A,0,1,\neg,\wedge,\vee)$ into a boolean
 algebra $A'= (A',0',1',\neg',\wedge',\vee')$ is a 
 function $\eta\colon A\to A'$ 
 that preserves the complemented lattice structure:
 $$\eta(\neg x)=\neg'x, \,\,\,\eta(x\wedge y)
 = \eta(x)\wedge' \eta(y), 
\,\,\, \eta(x\vee y)=
 \eta(x)\vee' \eta(y).$$
It follows that $\eta(0)=0'$ and $\eta(1)=1'$.
  
An {\em isomorphism}
of $A$ onto $A'$ is a one-one function  $\theta$
of $A$ onto $A'$ such that both $\theta$ and its inverse are homomorphisms.
 }
\end{definition}

$$
\fbox{Until further notice, $A
=  (A,  0,1 ,\neg, \vee,\wedge)$ 
denotes  a finite boolean algebra
}
$$

\medskip
\noindent
This restriction will be removed in Section 
\ref{section:infinite}.

\smallskip
Elements $a,b$ of a boolean algebra $A$ are said to
be {\it incompatible} if $a\wedge b=0.$

 \index{dominates}
\index{$x$ dominates $y$,\,\,\,$x\geq y$}

\smallskip

\begin{proposition}
\label{proposition:atomic1} Let   $x,y\in A.$

\smallskip

(i) If $x\not=0$ then $x$ 
{\em dominates} some atom  $b \in \at(A)$.  In symbols, 
    $x\geq b$.

 \smallskip   
    (ii) For every $a\in \at(A)$ exactly one of
$a\leq x$ or $a\leq \neg x$  holds.
 
  \smallskip     
    (iii) $x\leq y$\,\, iff \,\,$x\wedge \neg y=0.$

 \smallskip      
    (iv)   If $x  < y$ then   $y\wedge  \neg x\not=0$.

 \smallskip      
    (v)   $\neg x$ is the
largest  element of $A$  incompatible 
  with $x$.\,\,\,\,Symmetrically, $\neg x$ is the
the smallest  element  $z\in A$ such that 
 $x\vee z=1$.
 
   \index{incompatible elements of a boolean algebra}
  \index{incompatibility}

 \smallskip
 (vi)
 For a function $\phi\colon A\to B$ to be an
isomorphism of boolean algebras $A$ and $B$,
  it is necessary and sufficient that
$\phi$ is onto $B$ and  
for all $x,y\in A$,
\begin{equation}
 \label{equation:ncs}
x\leq y\,\,\, \mbox{  iff  }\,\,\,  \phi(x)\leq \phi(y).
\end{equation}
 \end{proposition}

\begin{proof}  
  (i) If $x$ is an atom we are done. 
 If $x$ is not an atom,
  by definition of minimality there is a nonzero
 $x_1< x$.
 If $x_1$ is an atom we are done. Otherwise
 there is nonzero $x_2<x_1$. Since $A$ is finite, 
any such  chain  $x>x_1>x_2>\dots$  will end
 after finitely many steps
 with a nonzero minimal  element. 
 
 \smallskip
 (ii)
If (absurdum hypothesis) $a\leq x$  and $a\leq \neg x$ 
then by definition of infimum, 
 $a\leq x\wedge \neg x =0$ which contradicts $0\not= a\in
 \at(A)$. Thus {\it at most one} of the inequalities
$a\leq x$ and  $a\leq \neg x$ holds. To show that
{\it at least one} inequality holds, let us assume
 $a\nleq x$, i.e.,   $a\wedge  x <  a $,  with the intent
 of proving $a\leq \neg x$. 
 Since $a$ is minimal nonzero,
 \begin{equation}
\label{equation:crisi1} 
 a\wedge  x =0.
    \end{equation} 
Furthermore,  
\begin{equation}
\label{equation:crisi2}
 a\wedge  \neg x \not=0.
   \end{equation} 
For otherwise,  combining 
$ a\wedge  \neg x =0$ with condition  (\ref{equation:crisi1})
and   the distributivity property of $A$,  we obtain
$
   0=(a\wedge  x)   \vee (a\wedge \neg x) =
   a\wedge (x\vee \neg x)= a \wedge 1 = a,
$
   which is impossible.
 
 By   (i) and (\ref{equation:crisi2}) there is
 an atom $b\leq  a\wedge  \neg x $, i.e.,
 $b \leq \neg x$ and    $b\leq a$.
It follows that  $b=a$, whence $a\leq \neg x,$ as desired.

  \smallskip
 (iii)
 $(\Rightarrow)$  From the assumption 
 $x\wedge y=x$ we have $x\wedge \neg y=
 x  \wedge y \wedge \neg y=0.$\\
  $(\Leftarrow)$
  We first write
  $
  y\vee(\neg y\wedge x)=(y\vee \neg y)\wedge (y\vee x)=
  1\wedge (y\vee x)= y\vee x.
  $
  From our standing hypothesis  $x\wedge \neg y=0$ we obtain
  $y\vee x=  y\vee(\neg y\wedge x)=
y\vee 0=y$ whence  $y\geq x.$

  \smallskip
 (iv)  By (iii), 
  $x\not= y$
   iff  \, either $x\wedge \neg y\not =0$ \,\,or\,\, $y\wedge \neg x\not=0$.
Therefore,   $x  < y$ iff     
 $y\not= x$ and $x \leq y$ iff  
  (either $x\wedge \neg y\not =0$ or $y\wedge \neg x\not=0$) and
  $x\wedge \neg y=0$ iff 
  $y\wedge \neg x\not=0$ and 
  $x\wedge \neg y=0$,
  which implies  
  $y\wedge \neg x\not=0$.

  \smallskip
 (v) Distributivity ensures that the supremum $s$ of 
 (the finite set of) all elements of $A$ incompatible
 with $x$ is also incompatible with $x$.
Therefore,   $s$ is the greatest element   of $A$
 incompatible with $x$.  In particular, since 
  $\neg x$ is incompatible with $x$,
 $s\geq \neg x.$  
 By way of contradiction,  assume  $s >\neg x$.
 Then from (iv) we can write
\,\,\, $
0\not=s \wedge \neg\neg x = s \wedge x=0,
 $
which is impossible.  So $s=\neg x.$

The rest is proved similarly.

\smallskip
(vi)
 If $\phi$ is an isomorphism then it is onto $B$,
 and  both $\phi$ and its inverse
preserve the lattice-order structure, whence (\ref{equation:ncs})
is satisfied.
Conversely, assume  $\phi\colon A\to B$ 
 is  onto $B$ and satisfies 
(\ref{equation:ncs}). It follows that $\phi$ is one-one.
  Further,  both $\phi$ and 
$\phi^{-1}$ preserve the lattice operations.
  By the characterization of $\neg x$ in (v),
$\phi$ and 
$\phi^{-1}$  also preserve the $\neg$ operation.
  \hfill{$\Box$}
 \end{proof}
  The symbol \,\,$\Box$\,\, stands for the end of a proof.

\index{$\Box$, the end-of-proof symbol}
 \index{$\emptyset$, the empty set}

\smallskip
Let   $\emptyset$ be shorthand for the empty set.   
 For any boolean algebra $A$  let 
\begin{equation}
\label{equation:pow}
 \mbox{$
\pow(\at(A))\,\,\,\,\,\,\,\mbox{ (read: ``the powerset of $\at(A)$'')}
$
}
\end{equation}
be the boolean algebra of all subsets of $\at(A)$,
where $0=\emptyset, \,\,\,1 = \at(A)$
and for all  subsets  
$X,Y$ of $\at(A)$,
$$
X\vee Y=X\cup Y, \,\,\,\,\,\,X\wedge Y= X\cap Y,\,\,\,\,\,\,
 \neg X=\at(A)\setminus X = \mbox{ the complement of }X.
$$

\index{$\pow(X)$, the powerset  algebra of a set $X$}

\medskip

 \index{$\downarrow x$, the atoms dominated by $x$}
  
 \begin{proposition}
\label{proposition:sum-of-atoms}
  For any $x\in A$ let\,\,
$
  \downarrow x
  $
  be  the set of atoms  dominated by $x$,
  $$
  \downarrow x=\{a\in \at(A)\mid a\leq x\}.
  $$
We then have

\medskip
(i) 
  $\downarrow y=\emptyset$ iff $y=0.$
 If\,\, $ \downarrow x=\{b_1,\dots,b_l\}\not=\emptyset$
then  $x=b_1\vee\dots\vee b_l$.

\medskip
(ii) The function \,\, $\downarrow$\,\,  
is an isomorphism of $A$ onto   $\pow(\at(A))$.
 \end{proposition}

 \begin{proof}  
 By    Proposition \ref{proposition:atomic1}(i),
 the  first statement is trivial.
For the second statement, by  definition of
 supremum,
  $ x\geq \bigvee_{i=1}^l b_i$.
 By way of contradiction,  assume   
 $x > \bigvee_{i=1}^l b_i.$ 
By Proposition     \ref{proposition:atomic1}(iv),\,\,\,
  $x\wedge \neg \bigvee_{i=1}^l b_i\not=0$,  whence
  by Proposition     \ref{proposition:atomic1}(i)
   some atom $b$ is dominated by
 both $x$ and $\neg \bigvee_{i=1}^l b_i.$  
 Since $\{b_1,\dots,b_l\}$ is the list of all atoms dominated
 by  $x$,  we may
 safely assume $b=b_1$. 
 Then $b_1\leq  \neg \bigvee_{i=1}^l b_i=\bigwedge_{i=1}^l \neg b_i$,
whence $b_1\leq \neg b_1.$ Thus 
 $b_1\leq b_1\wedge \neg b_1=0,$
which is impossible.

\medskip
(ii)   From (i) we have
 \begin{equation}
 \label{equation:downarrow}
      \mbox{  for any  } \{b_1,\dots,b_l\}\subseteq \at(A),\,\,\,\,\,\,\,
   \downarrow(b_1\vee\dots\vee b_l)=
   \{b_1,\dots,b_l\}.
   \end{equation}
Thus the  function $\downarrow$ is onto pow(at(A)).
In the light of    Proposition \ref{proposition:atomic1}(vi),
there remains to be proved  that 
$
\mbox{for all } x,y\in A,
 \,\,\,\,x\leq y\,\,\,\mbox{ iff }\,\,\,\downarrow x\,\subseteq\,\,\,   
  \downarrow y.
  $ 
 If   $x\leq y$   then every atom dominated
 by   $x$ is also dominated by   $y$, and hence, 
 $\downarrow x\,\,\,\subseteq\,\,\,   
  \downarrow y$. 
Conversely, if  $ x\nleq y$  then  $x\wedge \neg y\not=0$
  by  Proposition \ref{proposition:atomic1}(iii).
 Proposition \ref{proposition:atomic1}(i) yields
  an atom $a$ such that
$a\leq  x\wedge \neg y$.
Since  $a\leq \neg y$,   then
by Proposition \ref{proposition:atomic1}(ii),  
  $a\nleq y$, whence  $a\notin\,\, \downarrow y.$   
Since $a\leq x,$\,\,\,$a\in \,\,\downarrow x$. In conclusion, 
$\downarrow x\,\,\,\not= \,\,\,   
  \downarrow y.$
\hfill{$\Box$}
  \end{proof}
  
  \index{representation theorem}
 \index{theorem!powerset representation}

Part (ii) in Proposition 
\ref{proposition:sum-of-atoms}
is an  example of
a ``representation theorem''.
Its role is to give the reader a concrete realization of
  every finite boolean algebra $A$ as
  the boolean algebra of all subsets
  of    the set of atoms of $A$,
    equipped with union, intersection and complement.
 
\subsection{The geometry of finite boolean algebras in $\mathbb R^n$}
\label{section:geometry}

\index{geometry of finite boolean algebras}
\index{$\{0,1\}^n$, the vertices of the $n$-cube   $\interval^n$}
\index{$\interval^n$, the unit $n$-cube in $\mathbb R^n$}
\index{product!order}

As usual, $\mathbb R$ denotes the real line, $\mathbb R^2$
the cartesian plane, and 
 $\mathbb R^n$ the  $n$-dimensional euclidean space.
 We let $\interval^n$ denote
 the {\it unit $n$-cube} in $\mathbb R^n$. Then 
 $\{0,1\}^n$  is the set of vertices of $\interval^n.$
 The {\it standard basis vectors} of $\mathbb R^n$
 are denoted  $e_1,\dots,e_n.$
  Throughout this paper,  such basic facts as
 the linear independence of the basis vectors of $\mathbb R^n$
 will be used tacitly.

\index{the boolean algebra $\{0,1\}^n$}
 
\begin{definition} 
\label{definition:product-order}
{\rm The   {\em product order} 
 of  $\{0,1\}^n$  is defined by stipulating that 
 $x\leq y$ if and only if, whenever a coordinate $x_i$ of $x$
is $1$ then so is the  coordinate $y_i$ of $y$. 
}
\end{definition}

Geometrically,   $x <  y$ 
means that there is a path from $x$ to $y$ consisting of
steps along consecutive edges of the unit cube $\interval^n$,
where each step moves away  from the origin. 

\medskip
 A direct inspection yields:
 
   \index{standard basis vectors}
 \index{$e_1,\dots,e_n$, the standard 
 basis vectors of $\mathbb R^n$}

\begin{proposition}
\label{proposition:boolean-algebra}
The product order makes
  $\yesno^n$ into  a boolean algebra
  with  bottom element $0=$   the origin in $\mathbb R^n$ and
  top element  $1= (1,\dots,1)$.  For  any
$x,y \in \yesno^n$ the $i$th coordinate $(\neg x)_i$
of $\neg x$ equals  
$1-x_i$. Further,
$(x\wedge y)_i=\min(x_i,y_i)$ and
$(x\vee y)=\max(x_i,y_i).$  
 The atoms of $\yesno^n$ are the standard basis vectors
$e_1,\dots,e_n,$\,\,\,
$\at(\yesno^n)=\{e_1,\dots,e_n\}.$
\end{proposition}

\index{theorem!geometric representation of a finite boolean
algebra}

\begin{theorem}  
{\rm (Geometric representation 
of finite boolean algebras in $\mathbb R^n$)}
\label{theorem:geometric-isomorphism}
 With  $\{a_1,\dots,a_n\}=\at(A)$ and
$\{e_1,\dots,e_n\}$ the standard basis vectors in 
$\mathbb R^n$, 
let  the function
$\iota \colon A\to \yesno^n$  by defined  by
stipulating that  for all  $x\in A$,   
 \begin{equation}
  \label{equation:iota}
\iota(x)=\bigvee\{e_i\mid a_i\leq x\}
=\sum\{e_i\mid a_i\leq x\}\in  \mathbb R^n.
\end{equation}
Then $\iota$ is an isomorphism of 
$A$   onto    $\yesno^n$, 
$ {\iota}\colon 
A \cong \yesno^n.$
 \end{theorem}
 
 \begin{proof} By  (\ref{equation:downarrow}) and Proposition
  \ref{proposition:sum-of-atoms},\,\,the linear
  independence of the standard basis vectors
  ensures that  $\iota$  is one-one.
  To prove that every $v\in \yesno^n$  is in
  the range of  $\iota$,  let us write 
  \begin{equation}
  \label{equation:secchieta}
 v=e_{n_1}+\dots+e_{n_k}
 \end{equation}
 for a unique subset 
$\{e_{n_1},\dots,e_{n_k}\}$ of
$\{e_1,\dots,e_n\}$. Let $\{a_{n_1},\dots,a_{n_k}\}$ 
be the corresponding set of atoms of $A$, and 
$x_v=a_{n_1}\vee\dots\vee a_{n_k}$. 
By Proposition \ref{proposition:sum-of-atoms}(i),\,\,\,
$\downarrow x_v = \{a_{n_1},\dots,a_{n_k}\}$.
By (\ref{equation:iota}),\,\,
$\iota(x_v)=v,$
which shows that the function $\iota$ is onto
$\yesno^n.$

In view of  Proposition
 \ref{proposition:atomic1}(vi),
there remains to be proved that for all
$x,y\in A$\,\,\, $x\leq y\Leftrightarrow \iota(x)\leq
\iota(y).$  By  Definition \ref{definition:product-order}
and (\ref{equation:iota}),
 for all $i=1,\dots n$ we have
%
$$e_i\leq \iota(x)\Leftrightarrow a_i\leq x. $$
%
As a consequence, 
\begin{eqnarray*}
x\leq y
&\Leftrightarrow&
\,\downarrow x\,  \subseteq\,\,  \downarrow y,
\mbox{\,\, by Proposition \ref{proposition:sum-of-atoms}(ii)} \\
{}& \Leftrightarrow &
\mbox{for all }\,\, a\in \at(A)\,\, \mbox{ such that }  \,\,a\leq x
\,\,\mbox{ we have  }\,\,a\leq y\\
{}&\Leftrightarrow&
\mbox{for all}\,\,  e \in \at(\yesno^n) \mbox{ such that } e\leq \iota(x)
 \mbox{ we have } e \leq \iota(y)\\
{}& \Leftrightarrow&
\iota(x)\leq \iota(y).  \quad\quad\quad\quad\quad
\quad\quad\quad\quad\quad\quad\quad
\quad\quad\quad\quad\quad\quad\quad
\quad\quad\quad\quad\,\,\,\, \mbox{\hfill{$\Box$}}
\end{eqnarray*}
 \end{proof}
 
\index{transpose} 
 
 
 For every  (always column) vector $v\in \mathbb R^n$ we
let $v^T$  denote its transpose. Thus,  e.g., if
$v$ is the column vector   $x \choose y$
 then $v^T=(x,y)$.
For any vector $w$ in $\mathbb R^n$ we let 
$
v^T w 
$
denote the matrix multiplication of $v^T$ and $w$,
i.e., their scalar product in $\mathbb R^n.$

\smallskip
For any two (possibly infinite) sets  $X$ and $Y$, 
by a {\it one-one correspondence} between $X$ and $Y$
we mean an injective function $f\colon X\to Y$ 
with $f(X)=Y.$

 \index{column vector}
 \index{$v^T$, the transpose of $v$}
 \index{$v^T w$, the scalar product of $v$ and $w$}
 \index{scalar product} 
   \index{possible outcome of an event}
 
\smallskip
 We use the notation
$$
\hom(A)
$$
for  the set of  homomorphisms 
of $A$ into the two-element 
 boolean algebra
 $\yesno.$ 
It follows that any   $\eta\in \hom(A)$   satisfies the identities 
 $$\eta(\neg x)=1-\eta(x),\,\,\, \eta(x\wedge y)
 =\min(\eta(x),\eta(y)), 
\,\,\, \eta(x\vee y)= \max(\eta(x),\eta(y)).
$$
Intuitively, the  homomorphisms of $A$
into  $\yesno=\{no, yes\}$
are  the  ``possible outcomes'' of  the ``events''
(i.e., the elements) of $A$.
 
 The most general   example of a homomorphism of
   $A$ into $\yesno$ is given by the following result,
    which in general
   fails when $A$ is an infinite boolean algebra:
  
  \begin{corollary}
\label{corollary:tutto}
{\rm (Geometric representation of $\hom(A)$)}
For any  $A$ we have:

\smallskip
(i)  Let   the function
 $
a  \mapsto  \eta_a 
$
send  every atom $a$ into the  function 
$\eta_a\colon A\to \yesno$  such that
for all $x\in A$,
$
\eta_a(x)=1\,\,\, \mbox{ iff }\,\,\,  a\leq x.
$
 Then  $a\mapsto  \eta_a$
 is a one-one correspondence between $\at(A)$ and  
$\hom(A)$.
 The inverse correspondence sends 
any  $\eta\in \hom(A)$ to the only atom $a_\eta$ of
$A$ such that $\eta(a_\eta)=1.$

 \index{$\hom(A),$ the homomorphisms of $A$  into  $\yesno$} 

\medskip
(ii)
Let  $\{e_1,\dots,e_n\}=\at(\yesno^n)$.
For each $i=1,\dots,n$ let  the function
$\theta_{e_i}\colon \yesno^n\to \yesno$
be defined by  stipulating that for every  $v\in \yesno^n$ 
\begin{equation}
\label{equation:theta}
  \theta_{e_i}(v)= 1\,\,\,\mbox{   iff   }\,\,\, e_i\leq v
  \,\,\,\,\mbox{ in the product order    of }\yesno^n.
\end{equation}
Then the function  $e_i\mapsto \theta_{e_i},
\,\,\,(i=1,\dots,n)$
  is a one-one correspondence 
  between   $\at(\yesno^n)$ and 
$\hom(\yesno^n)$.  
For  any atom   $e$  of $ \yesno^n$,\,\,\,$\theta_e$ 
is the only homomorphism   of $\hom(\yesno^n)$
  such that $\theta_{e}(e)=1.$
The inverse correspondence sends
 every $\theta\in \hom(\yesno^n)$
to the uniquely determined standard basis vector 
$e_\theta \in \at(\yesno^n) \subseteq \mathbb R^n$ such
that $\theta(e_\theta)=1.$

\medskip
(iii)
For   every  $v\in \yesno^n$ and \,\,$i=1,\dots,n$, 
\begin{equation}
\label{equation:scalar}
\theta_{e_i}(v)=e_i^T v = \mbox{$i$th coordinate of v}.
\end{equation}
\end{corollary}

\begin{proof}
(i) 
 For all $x,y\in A$,   $\eta_a(x\wedge y)=1$
iff $x\wedge y\geq a$\,\, iff \,\,$x\geq a$ and $y\geq a$\, iff
\, $\eta_a(x)=1$ and $\eta_a(y)=1.$ Thus $\eta_a(x\wedge y)=
\eta_a(x)\wedge \eta_a(y)$. Similarly, $\eta_a$ preserves the
$\vee $ operation.  Preservation of
the   $\neg$ operation follows from Proposition \ref{proposition:atomic1}(ii). 
Therefore,  $\eta_a\in \hom(A)$.
 If $a$ and $b$ are distinct atoms of $A$ then $\eta_a(a)=1$
and  $\eta_b(a)=0$,  because $a\ngeq b$. Thus the function
$a\mapsto \eta_a$ is one-one. 
For  any   $\eta \in \hom(A)$ \,  let 
$B_\eta \subseteq \at(A)$ be the set of atoms  $b$ 
such that  $\eta(b)=1$.
$B_\eta$ is nonempty, for otherwise (absurdum hypothesis),
$$
\mbox{
 $1=\eta(1)=\eta(\bigvee_{i=1}^n a_i)
 =\max \{ \eta(a) \mid a\in \at(A)\}=0$,
 }
 $$
  which is impossible.
We  have just proved that the  function $a\mapsto \eta_a$ 
 is onto $\hom(A)$.
Finally, 
$B_\eta$  cannot contain two distinct atoms  $a,b$.
For otherwise,  (absurdum hypothesis),
Proposition \ref{proposition:atomic1}(ii)
yields  $b\leq \neg a$,  whence
  $0= 1-\eta(a)= \eta(\neg a) \geq \eta(b)=1$, a contradiction.
So precisely one atom $a_\eta$ belongs to $B_\eta.$  
Evidently, 
$\eta=\eta_{a_\eta}.$

\medskip

(ii)  By  Theorem
 \ref{theorem:geometric-isomorphism},
 this  is the special case of (i)  for  $A=\yesno^n$.

\medskip
(iii) For  arbitrary  $v\in \yesno^n$, let
$e_{n_1},\dots,e_{n_k}\in \mathbb R^n$ 
be the basis vectors $\leq v$.
By Proposition \ref{proposition:sum-of-atoms}(i)
and (\ref{equation:secchieta})
we have $v=e_{n_1}+\dots+e_{n_k}$. 
From (\ref{equation:theta})  we obtain
\begin{eqnarray*}
\theta_{e_i}(v)  = 1
 &\mbox{ iff } & 
 e_i\leq v,\,\,\,(\ref{equation:theta}) \\
  {} &\mbox{ iff }&
  e_i\in \{e_{n_1},\dots,e_{n_k}\},\,\,\, \mbox{by Definition
  \ref{definition:product-order}} \\
 {} &\mbox{ iff }&
 e_i^T (e_{n_1}+\dots+e_{n_k}) =1\\
  {} &\mbox{ iff }& e_i^T v=1.\quad\quad \quad \quad\quad \quad  \quad\quad \quad \quad\quad \quad\quad\quad \quad \quad\quad \quad \quad\quad \quad  \Box 
\end{eqnarray*}
\end{proof}

 \subsection{Events, atomic events, and possible worlds}
\label{section:recap}

 \index{event}

Fix an integer $n=1,2,\dots$ and a set $G=\{X_1,\dots,X_n\}$.
Does there exist a largest boolean algebra 
containing $G$ as a generating set?

\index{miniterm}
\index{independent!events}
\index{sequence of independent events}

Take a coin,  toss  it $n$ times,  and record
  the result (head=1 or tail=0) of each toss. 
Suppose for each $i=1,\dots,n,$\,\,  $X_i$ stands for
the event ``the result of the $i$th toss of my coin is head''.
With $\neg, \wedge, \vee$ the usual
connectives of boolean logic, let us consider the
$2^n$ boolean formulas (called {\it miniterms})
\begin{equation}
\label{equation:miniterms}
X_{1}^{\beta_{1}}\wedge \dots, \wedge\,
X_{n}^{\beta_{n}} 
\mbox{ 
with  $X_{i}^{\beta_i}=
X_{i}$ if $\beta_i=1$, and
 $X_{i}^{\beta_i}=\neg X_{i}$ 
if  $\beta_i=0$.
}
\end{equation}

 \noindent
These miniterms record any possible outcome of your
$n$ tosses of a coin.  Each miniterm stands for a 
sequence of $n$ ``independent''  events, in the sense that 
  the occurrence or non-occurrence of $X_i$
does not interfere  with the occurrence or
non-occurrence of $X_j$,\,\,\,($i\not=j$).

Let   $\mathsf F_n$ be the set of all boolean {\it formulas}
in the {\it variables} $X_1,\dots,X_n$, where two
formulas are identified iff they are logically equivalent:
thus for instance, $\neg\neg X_1$ is identified with $X_1$, 
  \,\,\,$X_1\wedge X_2$ is identified with $X_2\wedge X_1,$
  and $X_1\vee(X_2\wedge X_3)$ is identified with
  $(X_1\vee X_2)\wedge (X_1\vee X_3).$\footnote{By assigning
  0 or 1 to the variables in all possible ways and working in
  the two-element boolean algebra $\yesno$ one has a 
  familiar   mechanical
procedure to check if two formulas are logically equivalent.}

The result is the {\it free boolean algebra on the free
generating set $G=\{X_1,\dots,X_n\}$.}

Readers who (like de Finetti) have little propensity for logic
may adopt any of the following two 
alternative definitions of $\mathsf F_n$:

\index{boolean!formulas}
\index{boolean!formulas up to logical equivalence}
\index{logical!equivalence}
\index{free!boolean algebra $\mathsf F_n$}
\index{free!generating set}
\index{$\mathsf F_n$ the free $n$-generator boolean algebra}

\begin{quote}
$G$ generates $\mathsf F_n$ and 
 $
X_{1}^{\beta_{1}}\wedge \dots, \wedge\,
X_{n}^{\beta_{n}}\not=0$ for all 
$(\beta_{1},\dots\beta_n) \in \{0,1\}^n$;
\end{quote}
$$\mbox{or, equivalently,}$$
\begin{quote}
$G$ generates $\mathsf F_n$,   \,\,\,and for every
boolean algebra \,\,$A$\,\,\, and  {\it function} 
$f\colon G\to A$,\,\,\,$f$ uniquely
extends to a {\it homomorphism } of $\mathsf F_n$
into $A$.
\end{quote}
 

\medskip
\noindent
As expected, $\mathsf F_n$  is  the largest
possible boolean algebra containing $G$
as a generating set. For, if $F$ is another boolean algebra
 generated by $G$, then by our last definition of $\mathsf F_n$,
the identity function ${\epsilon}\colon X_i\mapsto X_i$ extends to a
homomorphism  $\tilde \epsilon$ of $\mathsf F_n$ into $F$.
 Now,   $\tilde \epsilon$ is 
onto $F$,  because the $X_i$ generate $F$. 
We conclude that  $|F|\leq |\mathsf F_n|$.

The miniterms of   $\mathsf F_n$  in (\ref{equation:miniterms}) 
are the  $2^n$ atoms  of $\mathsf F_n$.  
Intuitively, they are the ``atomic events'' of $\mathsf F_n$.
As in  Corollary \ref{corollary:tutto},
 each miniterm $t \in \mathsf F_n$   
uniquely determines a
 ``possible world'' of $\mathsf F_n$, i.e., a homomorphism
$\mathsf F_n$ of into $\yesno$, assigning
  1 or 0 to any event   $e \in \mathsf F_n$, according as
$e$ dominates the atom $t$ or is disjoint from $t$.
In Proposition \ref{proposition:atomic1}(ii) it is shown that
this alternative always occurs.
Conversely, any  $\eta\in \hom(\mathsf F_n)$ 
uniquely determines the atom $a_\eta$ given by
the smallest element $a\in  \mathsf F_n$
such that $\eta(a)=1.$ The independence of the events
$X_i$ results in the largest set of possible worlds. 
  \footnote{Kolmogorov calls each atom
  of $\mathsf F_n$  an ``elementary 
  event''. For Boole, the atoms of  $\mathsf F_n$ are
its ``constituents''.
 De Finetti says that each atomic event is a ``case''.
  of the dual space of $\mathsf F_n$.
  An infinite boolean algebra $A$ need not have
atoms.  The set $\hom(A)$  conveniently replaces
the set of atoms of $A$ in any case.} 

\index{Kolmogorov!elementary event}
\index{Boole!constituent}
\index{atom}
\index{atomic event}
\index{elementary event}
\index{constituent}
\index{sample point}
\index{possible world}
\index{homomorphism into $\yesno$}


\bigskip

Now suppose  $n=3$ and  $X_1,X_2,X_3 $ stand
for the following events, where
``wins'' is shorthand for
``wins the next 
FIFA Club World Cup'':
$$\mbox{
$X_1= \mbox{Brazil wins}$,\,\,\, 
$X_2=\mbox{ Spain wins}$,\,\,\, $X_3= \mbox{France wins}$.
}
$$
From the  rules of the FIFA World Cup
it follows that  not all $2^3$ atomic events
of $\mathsf F_3$,  coded by the miniterms  
(\ref{equation:miniterms}),   
can  occur. For instance, the atomic event
$X_1\wedge X_2\wedge \neg X_3$ is impossible,
and so is, a fortiori,  the atomic
event  $X_1\wedge X_2\wedge X_3$.
The impossible atomic events are precisely those
stating that the number of winners is $\geq 2.$  
On the other hand,
$\neg X_1\wedge \neg X_2\wedge \neg X_3$ is possible.
It follows that the boolean algebra $A$ generated by 
 the three events $X_1,X_2, X_3$   is 
 strictly smaller than  the free three-generator algebra 
 $\mathsf F_3$,  
  the largest possible boolean three-generator
  algebra.
\,\, We construct  $A$ by 
deleting  from the set of  atoms of
$\mathsf F_3$ those which  code 
atomic events  forbidden  by the rules of the FIFA cup.
A moment's reflection  shows that the surviving four 
atomic events  in  $A$ are as follows:
$$
\mbox{
$X_1\wedge \neg X_2\wedge \neg X_3$,
$\neg X_1\wedge  X_2\wedge \neg X_3$,
$\neg X_1\wedge \neg X_2\wedge X_3$,
$\neg X_1\wedge \neg X_2\wedge \neg X_3$.
}
$$
As expected,  the boolean algebra
 $A$ has  $2^4=16$ elements/events,  fewer than
the $2^8=256$ elements of $\mathsf F_3$.
Beyond  $X_1,X_2,X_3$ themselves,
 and the four atoms of $A$, 
examples of events of $A$  include    
 $\neg X_1\wedge   X_2$  (which in $A$ is the same  as
   $X_2$),  the impossible event
 $X_1\wedge X_2$  (i.e., the zero element of $A$), the sure event $X_1\vee \neg X_1$ (i.e., the top element 1 of $A$), and a few others. 
 
  As in Proposition 
 \ref{proposition:sum-of-atoms}(i), each event of $A$
 is the disjunction  (algebraically speaking, the supremum)
  of the atomic events it dominates.
 Thus, e.g., in $A$ we have
 $$\neg X_1  =
 (\neg X_1\wedge \neg X_2\wedge X_3)\vee
 (\neg X_1\wedge  X_2\wedge \neg X_3)
 \vee(\neg X_1\wedge \neg X_2\wedge \neg X_3),$$
 while an easy exercise shows that
  in  $\mathsf F_3$\,\,\,the element $\neg X_1$ is the
 supremum  of four atoms.
 
By  Proposition 
 \ref{proposition:sum-of-atoms}(ii), 
 up to isomorphism, every $n$-generator
  boolean
algebra  arises from a similar reduction procedure
of the atoms of a 
 free boolean algebra $\mathsf F_n$.
 
 \index{the most general finite boolean algebra}
 \index{finite!boolean algebra}

\subsection{States}
\label{section:states}
 
  \index{state}
  \index{state!additivity property}
   
\begin{definition}
\label{definition:state}
 {\rm  A {\em state}
 of a   boolean algebra 
 $A$ is a function $\sigma\colon A\to [0,1]$ with $\sigma(1)=1$,
  having the {\em additivity property}:
For all $x,y\in A$\,\,\,  if  $x\wedge y=0$  then 
$ \sigma(x\vee y)=\sigma(x)+\sigma(y).$
}
\end{definition}

\smallskip
\noindent
An easy verification shows that
{\it every homomorphism of $A$ into $\yesno$   is a state}.

\index{state!alias ``finitely additive probability measure''}
\index{probability!finitely additive measure}
\index{probability!Borel measure}
\index{state!of an ordered group with unit}
\index{state!of a C*-algebra}
\index{state!of a quantum system}
 
 \index{regular Borel probability measure}
 \index{Carath\'eodory extension  theorem}
 \index{theorem!Carath\'eodory extension}
 \index{finitely additive probability measure} 

 \medskip
 \begin{center}
  \fbox
{
\parbox{0.96\linewidth}
{
{\bf Historical/Terminological remark
 (for a second reading).}\\[0.1cm]
States of boolean algebras
 are also known as ``finitely additive probability
measures''. 
   Carath\'eodory
 extension  theorem  
yields an affine homeomorphism
of the space  of  states of every finite or infinite
boolean algebra $A$ onto the space of
 regular Borel  probability measures
on the Stone space of $A$.  
Hence  our terminology is preferable,
not only for its conciseness, but also to avoid confusion
between the finite additivity of states at the
 algebraic level of $A$ and the countable additivity of
Borel probability measures at  the topological
  level of the dual Stone space of $A$. The
  specific   choice of the
    term ``state'' rests on the
categorical equivalence $\Gamma$
  between MV-algebras 
    and unital $\ell$-groups,
  whose ``states'',  i.e.,
   unit-preserving monotone homomorphisms, 
   are deeply related to the
  ``states'' of  C*-algebraic quantum systems.
  For every unital $\ell$-group  $(G,u)$,
  setting $A=\Gamma(G,u)=[0,u]$,
  the  restriction function
$\sigma\mapsto \sigma\restrict [0,u]$
 is an affine homeomorphism of the state space of $(G,u)$
onto the state space of $A$.  Since boolean algebras
are precisely  idempotent MV-algebras, every state of
a boolean algebra $B$ uniquely corresponds to the only
  state of the unital $\ell$-group $(H,v)$
associated to $B$ by $\Gamma$,
and also determines  a  state of the
C*-algebra associated to $(H,v)$.
}
}
\end{center} 

\index{affine homeomorphism}
\index{unital $\ell$-group}
\index{state!of a unital $\ell$-group}
\index{unit-preserving monotone homomorphism}
\index{$\Gamma$, the equivalence of MV-algebras
and unital $\ell$-groups}
\index{boolean algebras = idempotent MV-algebras}

\medskip
 The  geometric representation
of every state $\sigma$ of a boolean
algebra $A$  is the object of  the
 following corollary, where $\sigma$ is
 identified 
 with a suitable vector in euclidean space, 
 and the value assigned by $\sigma$
 to an event is the scalar product of  the vector
 representing the state and  the vector
 representing the event.
 
  \index{event}
 
 \medskip
 A (possibly infinite) set $X$ in euclidean space
$\mathbb R^n$ is {\it convex}
if for any two 
points $x,y\in X$ the segment joining $x$ and $y$
is contained  in $X$.
 
 \index{convex!hull}
 \index{convex!combination}
 \index{convex!set}
\index{$\conv(e_1,\dots,e_n)$, the convex hull of 
$\{e_1,\dots,e_n\}$} 
 
 \medskip
 For any set $\{x_1,\dots,x_l\}\subseteq
 \mathbb R^n$, the {\it convex hull}  $\conv(x_1,\dots,x_l)$
 is the set of all {\it convex combinations} of 
 $x_1,\dots,x_l$, i.e., all points $x$ of $\mathbb R^n$
 of the form
 $$
 x=\lambda_1x_1+\dots+\lambda_nx_n,\,\,\,
\,\,\,0\leq \lambda_i\in \mathbb R,
 \,\,\,(i=1,\dots,n),\,\,\,
  \,\,\,\sum_{i=1}^n \lambda_i=1.
 $$
 An element $z\in X$ is {\it extremal} (in $X$)  if whenever
 $z\in \conv(a,b)$ for some $a,b\in X$ then
 $z=a$ or $z=b.$
  
 \index{extremal} 
\index{$\mathsf S$, the   convex hull of $e_1,\dots,e_n$}

\medskip 
\begin{corollary}
\label{corollary:vi-sigma}
{\rm (Geometric representation of the states of $A$)}
Let $\sigma $ be a state of the boolean
algebra $\yesno^n$. Let 
the closed convex set 
$\mathsf S\subseteq \mathbb R^n$
and the vector 
$v_\sigma \in \mathbb R^n$ be respectively defined by
\begin{equation}
\label{equation:vi-sigma}
\mathsf S=\conv(e_1,\dots,e_n)
\,\,\,\,\mbox{ and }\,\,\,\,
v_\sigma=(\sigma(e_1), \dots, \sigma(e_n)).
\end{equation}
Then
\begin{itemize}
\item[(i)]
$\,\,(v_\sigma)_1+\dots+(v_\sigma)_n=1$.
Thus  \,\,\,
$v_\sigma\in \mathsf S$.

	\medskip
\item[(ii)]$\,\,$ 
For all $ v=(v_1,\dots,v_n)\in \yesno^n,\,\,\,
\sigma(v) = v_\sigma^T v.
$
 
\medskip
 \item[(iii)]  $\,\,\,\,\,v_\sigma$ is the only vector in $\mathbb R^n$
 such that for all $v\in \yesno^n$,  $\sigma(v)=v_\sigma^T v.$
 Therefore, 
  any state is uniquely determined by the values it 
  assigns  to the atoms $e_i.$
\end{itemize}
\end{corollary}

\begin{proof}
 (i) The top element 1 in the boolean algebra $\yesno^n$ is the
 sum $e_1+\dots+e_n=e_1\vee\dots\vee e_n$. Since
 $e_i\wedge e_j=0$ for $i\not=j$ then $1$ is the supremum
 of incompatible elements of $\yesno^n.$ By additivity,
 $1=\sigma(\bigvee_{i=1}^n e_i)=\sum_{i=1}^n\sigma(e_i),$ 
 whence  
 $v_\sigma\in \mathsf S$. 
 
\medskip
 (ii) For each $i=1,\dots,n$ let us agree to write
 $0e_i=$ the origin in $\mathbb R^n$  and
 $1e_i=e_i.$ The vector  
  $v$ is a linear
 combination  of  the 
 pairwise  orthogonal vectors  $e_1,\dots,e_n$
  with uniquely determined coefficients 1 or 0
  according as $e_i\leq v$ or $e_i\nleq v$.
Further,
$v$ is  the supremum of the $e_i$ 
  with the same coefficients.
By 
Corollary \ref{corollary:tutto}(iii), 
 $
 v_i=\theta_{e_i}(v)\in \yesno,
 $
whence
 $$
 v=\sum_{i=1}^n\theta_{e_i}(v)e_i=
 \bigvee_{i=1}^n
\theta_{e_i}(v)e_i.
$$
By  (\ref{equation:vi-sigma})
and the  additivity property of $\sigma$ we can write
\begin{eqnarray*}
\sigma(v)
&=&
\sigma\left( \bigvee_{i=1}^n \theta_{e_i}(v)e_i\right)
=
\sum_{i=1}^n \sigma(\theta_{e_i}(v)e_i) =
\sum_{i=1}^n \theta_{e_i}(v)\sigma(e_i)
\\
{}
&=&
 \sum_{i=1}^n v_i(v_\sigma)_i =
v^Tv_\sigma
=
v_\sigma^Tv.
\end{eqnarray*}

\smallskip
(iii) Suppose $w\in \mathbb R^n$ satisfies
$w^T v=\sigma(v)$ for all $v\in \yesno^n.$ Then by
 (i)-(ii),  for all $i=1,\dots n,$ we have
$(w^T-v_\sigma^T)e_i=0$ whence  $w_i=(v_\sigma)_i$
and $w=v_\sigma.$
\hfill{$\Box$}
\end{proof}

\subsection{De Finetti's  ``Fundamental Theorem of Probability''}
\label{section:fundamental}

In  his  ``Investigation of the laws of thought", 
  \cite[Chapter XVI, 4, p. 246]{Boole1854},
 \,\,\, Boole writes:
  
  \index{Boole}
  \index{Boole!the object of probability}
  \index{Boole!probabilistic inference problem}
  \index{the object of probability theory}
  \index{probability!its object}

\smallskip
\begin{quote}
{\small ``the object of the theory of 
probabilities might be thus defined.   
Given the probabilities of any events,  of whatever kind, 
to find the probability of some other event connected with them."
}
\end{quote}

\smallskip
\noindent
In his paper \cite{def-fundmath}
de Finetti  gave a  criterion for
the  probability of the new  event to exist.
This is his consistency theorem \ref{theorem:dutch}, 
based on Definition \ref{definition:consistency}.
Furthermore, in 
  \cite[p. 13]{def-poincare}  and on page 112 of  his book
\cite[3.10.1]{def-wiley1},  
de Finetti showed that 
the set of  possible probabilities of the new event
(if nonempty) is a closed interval contained in $\interval.$

 \smallskip
He named his result
``the Fundamental Theorem of Pro\-b\-a\-bi\-l\-ity''.

 \medskip
 \noindent
For our  self-contained proof of this theorem
we prepare the following notation 
and terminology, which will also be used  in a later chapter: 
 A (closed) {\it hyperplane}
 \index{hyperplane}
 $H$ in euclidean space $\mathbb R^n$ is a
 subset of $\mathbb R^n$ of the form
 $\{x\in \mathbb R^n\mid l^T x = w\}$ for some
 nonzero vector $l\in \mathbb R^n$ and
 real number $w$. The sets
  $\{x\in \mathbb R^n\mid l^T x \geq  w\}$ 
   and 
    $\{x\in \mathbb R^n\mid l^T x \leq w\}$  
    are the two {\it (closed) half-spaces}
    \index{half-space}
        \index{closed half-space}
with boundary  $H$.

\medskip 
 
\index{de Finetti!``Fundamental Theorem of Probability''} 
 \index{theorem!de Finetti ``Fundamental Theorem of Probability''}
 \index{$S(A)$, the state space of $A$}

\begin{theorem}
{\rm (De Finetti's ``Fundamental Theorem of Probability'')}
\label{theorem:fundamental}
Let $A$ be a boolean algebra with $n$ atoms 
and state space $S(A)$.
Fix a finite subset   $E=\{h_1,\dots,h_m\}$  
of $A$ along with  a function 
 $\beta\colon E\to \interval$, and let 
$$
S_\beta=\{\sigma\in S(A)\mid \sigma\supseteq \beta\}
=\mbox{be the set of states of $A$  extending $\beta$.}
$$
Next,  for  any $h\in A$ let 
\begin{equation}
\label{equation:sbh}
S_\beta(h)=\{\sigma(h)\mid \sigma\in S_\beta\}
\end{equation}
be the set of possible values assigned
 to  $h$ by all states extending
$\beta.$
Then  $S_\beta(h)$ is either empty, or is a  
closed interval  contained in $\interval$
possibly consisting of a single point.
\end{theorem}

\begin{proof}  In view of the representation Theorem
\ref{theorem:geometric-isomorphism},
we can identify $A$ with the boolean algebra
$\yesno^n.$ Then $E$ is  
  a set of  vertices  $h_j\in \yesno^n$ of the
$n$-cube $\interval^n$. For all
$\sigma\in S(A)$ and $w\in \yesno^n$,\,\,\,
the value $\sigma(w)$ is given by the
scalar product $v_\sigma^T w,$ with 
$v_\sigma=(\sigma(e_1), \dots, \sigma(e_n))$
the  vector associated to $\sigma$ by the
geometric representation in
Corollary \ref{corollary:vi-sigma}. 
In particular, 
$v_\sigma$  lies  in the closed convex  set  
$$\mathsf S=\conv(e_1,\dots,e_n),
\mbox{
with $e_1,\dots,e_n$ the standard basis 
vectors of $\mathbb R^n$.}
$$
 For every  $j=1,\dots,m$,  the set    
$\{\sigma\in S(A)\mid
\sigma(h_j)=\beta(h_j)\}$ of states of $A$ extending
the singleton function $h_j\mapsto \beta(h_j)$ 
is  the set of vectors 
$$\{v_\sigma \in \mathsf S\mid \sigma(h_j) =\beta(h_j)\}
=\{v_\sigma\in \mathsf S\mid
v_\sigma^T h_j=\beta(h_j)\}.$$
 This is  the intersection  of $\mathsf S$
 with the hyperplane
   $H_j\subseteq \mathbb R^n$
of all vectors $w\in \mathbb R^n$ such that
$w^Th_j=\beta(h_j).$ 
As a consequence,  the set
$S_\beta\subseteq \mathsf S$
  of states of $A$ extending $\beta$ coincides with
$
 \mathsf S\cap H_1\cap\dots\cap H_m.
$
It follows  that $S_\beta$ 
 is a closed   convex (a fortiori connected) subset of 
$\mathsf S$. 
The set $S_\beta(h)$  in \eqref{equation:sbh}
has the form
$ \{v_\sigma^T  h\in \interval \mid    \beta
\subseteq \sigma \in \mathsf S\}.$ 
Thus $S_\beta(h)$   is the range of the
  continuous  $\interval$-valued function
$v_\sigma\mapsto v_\sigma^T  h$,
where the vector
$v_\sigma$ ranges over the
  closed connected space  $S_\beta$. 
Elementary topology
shows  that $S_{\beta}(h) $  is closed and connected.
In conclusion,   if nonempty, $S_{\beta}(h)$
 is a closed  interval in $\interval$.
\hfill{$\Box$}
 \end{proof}

 \index{PSAT!probabilistic satisfiability}
 \index{PSAT!and Boole's probabilistic entailment problem}
 \index{PSAT!optimization version}
 \index{probabilistic satisfiability}
 \index{optimization version of PSAT}
 \index{boolean!formulas}
  \index{boolean!formulas as codes}
 
  \begin{remarks}
  \label{remark:PSAT}
  {\rm
  Boole's remarks on the object of the theory of
  probabilities are taken up  today in  the
  PSAT  (probabilistic satisfiability)  problem and its optimization version. Here events $h_1,\dots,h_m$ are assigned
 probabilities $p_1,\dots, p_m,$
  and one has to specify the possible values of
  the probability
 $p_{m+1}$  of a new event $h_{m+1}$. 
 For computational purposes, all  events $p_i$ are coded by {\it boolean formulas}, and all $p_1,\dots, p_m,$
are  {\it  rational} numbers.

 Theorem \ref{theorem:fundamental}
  shows that the set 
  of possible values of   $p_{m+1}$, 
  if nonempty, is  a closed interval contained in  $\interval.$
To  provide a fundamental necessary and sufficient condition 
for  this set  to be nonempty, it will take de  Finetti's consistency 
($\sim$ Dutch book)  theorem \ref{theorem:dutch}.
 PSAT is the problem of checking if  this condition
 is valid. PSAT  is
 a generalization of the 
  satisfiability problem SAT for boolean formulas, and has 
  the same computational complexity as SAT. 

PSAT and  its optimization  version
  pertain to a vibrant research area 
in various domains, including 
defeasible reasoning, 
automated deduction,  
formal epistemology, 
 and  uncertainty management.
For more information, also including de Finetti's contributions to
Boole's problem and  its modern reformulation, see  
   {P. Hansen},   {B.  Jaumard},
``Probabilistic Satisfiability'',
in: Handbook of Defeasible
Reasoning and Uncertainty Management
Systems, Vol. 5.
(J. Kohlas, S. Moral, Editors).
Springer, 2000, pp. 321-367.
}
  \end{remarks}

\section{De Finetti's Consistency   Theorem}
\label{chapter:consistency}
 
  \index{lottery}
  \index{lottery!one-event ticket}
 
Lot T. Tery is an honest and experienced manager of a
worldwide  lottery.  Each ticket reads:

 \medskip
\begin{center}
\fbox
{
\parbox{0.52\linewidth}
{
\footnotesize
I, the undersigned seller of this ticket, will buy it back paying the bearer one euro,  if \,Spain wins the next FIFA world cup.

\vspace{0.1cm}

\hfill{(Signed: the seller)}
}
}
\end{center}

\noindent
The price of each ticket is $p$ (euro, to fix ideas). 
 
  If   $p$ is a fair price for you and you decide to buy $N$ tickets, you now pay $Np$    to Lot T. Tery.
  He, the signatory ticket seller, will   pay you 
   $N$  if Spain wins.
 
On the other hand,  if  the ticket price seems too high to you, why not ask  Lot T. Tery   to {\it buy}  $M$  tickets from you?
  What better proof that the  price $p$ is right?
His acclaimed honesty  makes him willing to
exchange  his managerial role with you: 
  he  pays you   $pM$ for  $M$   tickets.
You, the signatory seller of the ticket, 
 will pay him  $M$  if Spain wins.

 As de Finetti  notes on   page 309  of
 his paper  \cite{def-fundmath}, 
   the price $p$ set by any experienced manager  like 
   Lot T. Tery must satisfy the 
trivial  inequalities   $0 \leq p\leq 1$.  Otherwise,  you could 
 bankrupt him, whatever the outcome of  the next FIFA world cup.  As a matter of fact, if $p>1$,  
  selling  him  $Z$ tickets,
 with $Z=$ one zillion,  you will get  $(p-1)Z$  
  if Spain wins  and $pZ$  otherwise. 
Lot T. Tery will get even worse   if he sells 
$Z$ tickets at a   price  $p<0$.  

Therefore, first of all, consistency requires that 
the ticket price be in the range   $\interval.$
Any inconsistency on  Lot T. Tery's part
can be punished by you,  sending  him to ruin, whatever
the outcome of the FIFA cup.

  \index{lottery!consistent ticket price} 
 
What other conditions must  the ticket price $p$  
meet, to keep you from bankrupting Lot T. Tery? 
Definition \ref{definition:consistency} provides  the answer,
even  for the general case where  tickets  for multiple events
and sold/bought.
De Finetti's consistency theorem \ref{theorem:dutch}
gives a characterization of consistency.

\subsection{Consistency}
\label{section:consistency}

\index{consistency}
\index{consistent!book}
Following de Finetti,   we
   define the fundamental notion of a
 consistent   
 assignment of numbers in $\interval$  to 
 ``events''   $h_1,\dots h_m$ understood as elements of  a
finite  boolean algebra $A$.  The finiteness hypothesis
will be dropped in Section 	\ref{section:infinite}.
 
 \smallskip
With  $\{a_1,\dots,a_n\}$ the set of
atoms of  $A$,  let
$\{\eta_{a_1},\dots,\eta_{a_n}\} =
\{\eta_{1},\dots,\eta_{n}\}=\hom(A)$ be the set of  
 homomorphisms of $A$
into the two-element boolean algebra $\yesno$, as in 
 Corollary \ref{corollary:tutto}(i).  

\medskip
\begin{definition}
{\rm (De Finetti,
\cite[pp. 304-305]{def-fundmath})}
\label{definition:consistency}
{\rm 
\,\,\,Let $E =\{h_1,\dots h_m\}$  be a subset of  $A$ and
  $\beta \colon E \to   \interval$  a function.
  Then $\beta$ is said to be {\em inconsistent in $A$}
  \index{inconsistent book}  if
 there exists a function   $ s\colon  E \to \mathbb R$ such that 
\begin{equation}
\label{equation:consistency-finite}
\sum_{j=1}^m s(h_j)(\beta(h_j)-\eta_i(h_j)) < 0
\,\,\mbox{ for every }\,\,  i=1,\dots,n.
\end{equation}
If $\beta$ is not inconsistent in $A$  we say it is {\em consistent
in $A$}.
}
\end{definition}

\index{consistent!book}


\index{betting odd}
\index{betting rate}
\index{probability!as a price}
\index{probability!as a prevision}
\index{stake}

To better understand  Definition \ref{definition:consistency}
and its characterization in  theorem \ref{theorem:dutch}, 
recalling
de Finetti's second citation at the outset of Chapter
\ref{chapter:fundamental},
 let us replace  Lot T. Tery 
  by a bookmaker  named Bookie 
grappling with a clever
bettor  named Betty. 
Bookie posts  ``betting odds'', or ``betting rates''
$$p_1=\beta(h_1),\dots,p_n=\beta(h_m)
\in \interval$$
 on future  events
\,\,$h_1,\dots,h_m\in A$.
As in the example above,   
 $p_j=\beta(h_j)$   is
  the price of each ticket  for a
payoff  of  1 euro   in case $h_i$ 
  occurs, and 0 otherwise.\footnote{Real-life bookmakers
  prefer to post their odds as  $1/p_j\geq 1$
   instead of $p_j$  (always $>0$), guess why.} 
 
 \smallskip
If Bookie's price  $\beta(h_j)$ is deemed
reasonable by Betty,
she   now places a stake $s(h_j)\geq 0$,
 pays  $\beta(h_j)s(h_j)$, 
hoping to win   
$s(h_j)$\,  if  $h_j$ occurs.
It goes without saying that if $h_j$ does not occur
Betty's win will be zero.

\index{swapping the bookie/bettor roles}
 \index{bettor-to-bookie orientation}
\index{balance}
\index{real-life bookmakers}

On the other hand, if  Betty finds  
$\beta(h_j)$  excessive\,\,---\,\,as is almost always the case 
 with bookmakers' odds in real life\,\,---\,
 {\it she can place a stake $s(h_j)<0$.}  So  Bookie 
pays   Betty  $\beta(h_j)\,|s(h_j)|$,  and wins 
 $|s(h_j)|$    if  $h_j$ occurs.  
In this case,  Bookie/Betty's  roles swap.
This never 
happens in real life.

For definiteness, let us make the following stipulation:
\begin{equation}
\label{equation:orientation}
\mbox{\small  All financial  transactions are given the 
 Betty-to-Bookie orientation.}
 \end{equation}
Then for  any positive or negative  stake $s(h_j)$,
the balance of this single  bet  on event $h_j$
is given by 
$s(h_j)(\beta(h_j)- \eta_i(h_j))
\,\,\,\mbox{ in the ``possible world'' }\,\,\eta_i\in \hom(A).$
\footnote{It is understood that
  $\eta_i(h_j)$ equals 1 or 0 according
as $h_j$ occurs or does not occur in  $\eta_i$.}
\,\,\,Since zero bets are possible,
we may assume Betty is betting on all events $h_1,\dots,h_m$.
Bookie's  balance is then  
\begin{equation}
\label{equation:balance of gambler's transaction}
\sum_{j=1}^m s(h_j)(\beta(h_j)- 
\eta_i(h_j)). 
\end{equation}
Bookie's book $\beta$ is   inconsistent
 in the sense of
Definition \ref{definition:consistency}  iff  Betty  
can devise  positive or negative 
stakes  $s(h_1),\dots,s(h_m)\in \mathbb R$ 
that guarantee  her  a minimum profit of  1
\,\,---\,equivalently, one zillion\,---\,  
``re\-gard\-less of the outcome of the events  $h_j$'',
i.e.,  in any ``possible world''     $\eta_i\in \hom(A)$.
Correspondingly, 
 Bookie has a sure loss of at least one(zillion).
 
 \index{event}
 \index{de Finetti!event} 
 \index{inconsistent book} 
 \index{future outcome}
\index{possible world}

The totality  of
``possible outcomes", or   ``possible worlds'',
is made precise by the homomorphisms
$\eta_1,\dots,\eta_n \in \hom(A)$.  
Definition \ref{definition:homomorphism}
ensures  that
the laws of logic hold in each  $\eta_i$.

To view  panoramically   all events
$h_j,\,\,\,(j =1,\dots,m)$ and all possible worlds  $\eta_i,
\,\,(i=1,\dots,n)$,
we prepare  the following
$n\times m$ matrix $M$:

\medskip
\begin{equation}
\label{equation:emme}
M= 
\left(
\begin{array}{cccc}
 \vspace{0.3cm}
\beta(h_1)-\eta_1(h_1)&\quad \beta(h_2)-\eta_1(h_2)&
\,\,\dots\,\,&\beta(h_m)-\eta_1(h_m)
\cr \vspace{0.3cm}
\beta(h_1)-\eta_2(h_1)&\quad\beta(h_2)-\eta_2(h_2)&
\,\,\dots\,\,&\beta(h_m)-\eta_2(h_m)
\cr \vspace{0.3cm}
\dots& \,\dots\,\, & \,\,\dots\,\,&\dots\cr
\beta(h_1)-\eta_n(h_1)&\quad\beta(h_2)-\eta_n(h_2)&
\,\,\dots\,\,&\beta(h_m)-\eta_n(h_m)
\end{array}
\right)
\end{equation}

%
%

 \medskip
\noindent
Let  $s(h_1),\dots,s(h_m)\in \mathbb R$ be
Betty's  stakes. Fix a row   $i = 1,\dots,n$ and a column 
$j  =1,\dots,m$.\,\,\,
 
If $\eta_i(h_j)=0$, i.e., if  $h_j$ does not occur in the
possible world  $\eta_i$,  then the term  \,\,$M_{i,j}$ \,\,in $M$
coincides with    $\beta(h_j)$.  
Recalling \eqref{equation:orientation},
Bookie's  profit/loss  is  $\beta(h_j)s(h_j)\in \mathbb R$.
On the other hand,  if  $h_j$ occurs in the
possible world  $\eta_i$ 
then 
$M_{i,j}=\beta(h_j)-1$  and
Bookie's   profit/loss  is 
$(\beta(h_j)-1)s(h_j)$.    
In either case,  
Bookie's profit/loss for Betty's  stake  $s(h_j)$ on
 event $h_j$  is  
$(\beta(h_j)-\eta_i(h_j))s(h_j)$
 in  the possible world
  $\eta_i.$

Let  now
$s\in \mathbb R^m$ 
be the (always column) vector whose  coordinates
are $s_1=s(h_1),\dots,s_m=s(h_m)$.
Then the  vector $Ms\in \mathbb R^n$
gives the balance 
(\ref{equation:balance of gambler's transaction})
in all possible worlds $\eta_1,\dots,\eta_n$.
Bookie's $\beta$ is inconsistent iff Betty can devise
a vector  $s\in \mathbb R^n$  (equivalently,  $s\in \mathbb Q^n$)  such that 
all coordinates of  $Ms$ are $<0.$

 \subsection{The Consistency  (alias Dutch book)  Theorem}
 \label{section:consistency-theorem}

  \index{Dutch book theorem}
 \index{de Finetti!consistency theorem}
 \index{de Finetti!Dutch book theorem}
\index{theorem!Dutch book}
\index{theorem!de Finetti consistency theorem for finite
boolean algebras}
\index{consistency}
\index{consistency!=extendability to a state}

\begin{theorem} 
{\rm (De Finetti's consistency   theorem,
\cite[pp. 309-313]{def-fundmath})}
\label{theorem:dutch}\,\,\,
For any boolean algebra $A$
with atoms  $a_1,\dots,a_n$ and corresponding
homomorphisms $$\eta_1,\dots,\eta_n$$ of $A$  into
$\yesno$,   
let $E=\{h_1,\dots,h_m\}\subseteq
A$ and $\beta\colon E\to \interval$. Then 
precisely one of the following conditions holds:

\smallskip
\begin{itemize}
\item[(i)] 
$\beta$ is   inconsistent in $A$ in the sense of
Definition \ref{definition:consistency}.  Thus
 there is a function   $ s\colon  E \to \mathbb R$ such that,
 letting   $s_1=s(h_1),\dots,s_m=s(h_m)$,
 every coordinate of the  vector
$M(s_1,\dots,s_m)\in \mathbb R^n$ is $<0$. 
In symbols, 
$$
\sum_{j=1}^m s(h_j)(\beta(h_j)-\eta_i(h_j)) < 0\,\,
\mbox{ for every }\,\,  i=1,\dots,n.
$$
 \item[(ii)]  $\,\,\beta$ is  extendable to a  state $\sigma$ of $A$.
  \end{itemize}

 \medskip
 \noindent
Stated otherwise,  $\beta$ is consistent in $A$
 iff  it is extendable to a state    of $A$.
\end{theorem}
\begin{proof}  In view of Theorem 
\ref{theorem:geometric-isomorphism},
by identifying
$A$ with the boolean algebra $\yesno^n$ of the
vertices of the $n$-cube $\interval^n$, each atom 
$a_i$ is geometrically realized as the 
 $i$th standard   basis vector  $e_i$ in $\mathbb R^n.$
Furthermore, each element  $h_j$ is realized as
  a vertex of the $n$-cube $\interval^n.$
Since the isomorphism $\iota$ 
of Theorem \ref{theorem:geometric-isomorphism}
now coincides with  identity, by 
 Corollary \ref{corollary:tutto}(ii) we can write
\begin{equation}
\label{equation:uguaglianze}  
\mbox{$\eta_i=\theta_i=\theta_{e_i}\in \hom(\yesno^n)$
for each
$i=1,\dots,n$.
}
\end{equation}

\medskip
\noindent
{\it Proof that if condition (i) fails then
 condition  (ii) holds.}
  \index{Gordan's theorem}
 \index{theorem!Gordan}
Suppose condition (i)  fails. 
  Gordan's theorem\footnote{A self-contained   proof
  of Gordan's theorem is given  in  \ref{theorem:gordan}.} 
  gives a nonzero column vector $(u_1,\dots,u_n)\in \mathbb R^n$
  with each coordinate $u_i\geq 0$ and $u^TM=0=$
  the origin in $\mathbb R^m.$ 
  Without loss of generality, $u_1+\dots + u_n=1$.
Summing up, 
   \begin{equation}
   \label{equation:rombo}
 u^TM=0, \,\,\,\,\,\,\,0\leq u_i\,\,\,
(i=1,\dots, n),\,\,\,
\,\,\,\,u_1+\dots + u_n=1.
   \end{equation}
For  each  $j=1,\dots,m$\,\,  
the $j$th column  $C_j$  of 
the $n\times m$  matrix $M$  in  (\ref{equation:emme})
has the form 
$C_j=C_{j,{\rm left}}-C_{j,{\rm right}},$
where  
  \begin{equation}
 \label{equation:left}
 C_{j,{\rm left}}=\underbrace{(\beta(h_j),
 \dots,\beta(h_j))}_{\mbox{\tiny $i$  times}}.
 \end{equation}
For each
  $i=1,\dots,n$,   the $i$th term $(C_{j,{\rm right}})_i$  
satisfies the identities 
  \begin{equation}
 \label{equation:right}
 (C_{j,{\rm right}})_i=\eta_i(h_j)=\theta_i(h_j)=
e_i^Th_j=i \mbox{th coordinate of }h_j\in \yesno^n.
\end{equation}
This follows from  
\ref{theorem:geometric-isomorphism}-\ref{corollary:tutto}  and 
\eqref{equation:uguaglianze}.\,\,
By \eqref{equation:rombo}-\eqref{equation:right}, the
 assumption  $u^TM=0$, \,\,\,
 (i.e., $u^TC_{j,\rm right}=u^TC_{j,\rm left}$ for each
 $j=1,\dots,m$)\,\,\, 
 amounts to writing 
\begin{equation}
\label{equation:lurena}
\beta(h_j)=u^T C_{j,\rm right}
=u^T h_j, \,\,\,\mbox{for each}\,\,j =1,\dots,m.
\end{equation}
Let the state $\sigma$
of the boolean algebra $\yesno^n$ be defined by 
$$
\sigma(e_i) =u_i,\,\,\,(i=1,\dots,n).
$$
By Corollary \ref{corollary:vi-sigma}(iii) and 
 (\ref{equation:lurena}),
$
\sigma(h_j)
= u^T h_j
 =\beta(h_j)
$
for each $j=1,\dots,m.$
This shows that $\sigma$ extends $\beta$, whence
condition (ii) holds.  

\medskip
\noindent
{\it Proof that conditions (i) and (ii) are incompatible.}
\,\,\,By way of contradiction, assume both conditions hold.
As  in   Corollary \ref{corollary:vi-sigma}, let 
   the vector  $v_\sigma\in \mathsf S$
 be defined by   $v_\sigma^T w =\sigma(w)$ for all 
 $w\in \yesno^n$. 
In  particular, 
$v_\sigma^T h_j= \sigma(h_j)$ for each 
$j =1,\dots,m$. Since all coordinates of $v_\sigma$ are $\geq 0$
and their sum is equal to  1, then
 $$v_\sigma^TC_{j,\rm left}
=\beta(h_j) \,\,\,\mbox{ for all }\,\,\, j =1,\dots,m.$$
 Further,  by (\ref{equation:right}),
$v_\sigma^TC_{j,\rm right}=v_\sigma^Th_j=\sigma(h_j)$.
Since, by condition  (ii),  $\sigma$ extends $\beta$ then
 $$v_\sigma^TC_{j,\rm left}
=v_\sigma^TC_{j,\rm right}
\,\,\, \mbox{ for all }\,\,\, j=1,\dots,m,$$
whence $v^T_\sigma M=0.$ 
As in   Gordan's theorem \ref{theorem:gordan},
conditions (i) and (ii) imply  the   contradiction
$$
0\,\,\, \not=\,\,\,
 \underbrace{{\,\,\,}v_\sigma^T{\,\,\,}}_{v_\sigma \in \mathsf S}\,\,\,
   \underbrace{ (M(s_1,\dots,s_m))}_{ s_j<0,\,\, 
   \mbox{\tiny for each }j =1,\dots,m}
=\,\,\, (v_\sigma^T M)(s_1,\dots,s_m)\,\,\,=
\,\,\,0. 
$$
We have thus shown that conditions (i) and (ii) are incompatible.

\medskip
\noindent {\it Conclusion.}
If $\beta$ is consistent then condition (i) fails, whence,  by
condition (ii), $\beta$ is extendable to a state.  Conversely, if condition (ii) holds then
$\beta$ cannot satisfy condition (i), i.e., $\beta$ is consistent.
\hfill{$\Box$}
\end{proof}

\index{de Finetti!event} 


\subsection{Interlude: Logical consistency and 
 de Finetti's consistency}
\label{section:digression}
\index{infinite boolean algebra} 
This section requires some  acquaintance with
the syntax and semantics of boolean propositional logic:
formulas, truth-valuations, logical equivalence and
consistency. 
The first pages of any book on mathematical
 logic will provide all necessary background.
Readers not interested in 
the relationships between logical consistency and de
Finetti's consistency can  skip this section 
on a first reading, as
no results proved here will be used  in the rest of 
this paper. 

 \index{logically consistent set} 
 \index{free!boolean algebra $\mathsf F_n$}
 \index{$\mathsf F_n$ the free $n$-generator boolean algebra}
 \index{boolean!formulas as codes}
 \index{boolean!formulas up to logical equivalence}

Let  $\mathsf F_n$ be the free boolean algebra
   over the free generating set $\{X_1,\dots,X_n\}$
   introduced  in
Section \ref{section:recap}.
   Elements of    $\mathsf F_n$ are boolean formulas
 $\phi$  in the variables $X_1,\dots,X_n$ 
   up to logical equivalence. The
   boolean operations naturally
   act on  $\mathsf F_n$ as the boolean
   connectives  act on formulas.
Thus,  e.g., let 
   $\phi$ be a formula  and 
   $\tilde\phi\in \mathsf F_n$
   the (infinite!) set of all formulas logically equivalent to $\phi$.
   Then     for any formula $\psi$ logically equivalent to $\phi$
   the equivalence class $\widetilde{\neg\psi}$
    of  $\neg\psi$ coincides with
    the complement  $\neg\tilde\phi\in  \mathsf F_n$  of the
    equivalence class $\tilde\phi$.   
Any  formula $\psi$ logically equivalent to $\phi$ is said to 
{\it code} the element  $\tilde\phi$ of $ \mathsf F_n$.
    
For each  $i=1,\dots,n,\,\,\,$ 
$X_i$   is a special kind
    of a formula, known as a {\it variable}.
Traditionally,
 the same notation is used for  $X_i$ and for
 the  equivalence class  $\widetilde{X_i}\in  \mathsf F_n$.
  Likewise,  $\neg,\wedge, \vee$ denote both
  the connectives acting on formulas, and the
  operations of the boolean algebra $ \mathsf F_n$.

A  {\it truth-valuation}
is a  $\yesno$-valued function $v$ defined on
  all formulas  $\psi=\psi(X_1,\dots,X_n)$, 
  having the following properties:
  $$
  v(\neg\psi)=1-v(\psi),\,\,\,v(\psi\wedge \phi)=\min(v(\psi),
  v(\phi)),\,\,\,v(\psi\vee \phi)=\max(v(\psi),
  v(\phi)).
  $$
  The non-ambiguity of
  the syntax ensures that $v$ is uniquely determined
  by the values it assigns to the variables   $X_1,\dots,X_n$.
  
\index{variable}
 \index{free!generator}   
 \index{truth-valuation}
 \index{logical!consistency}

%

Let the  formulas
 $\phi_1,\dots,\phi_m$ respectively code elements
  $h_1,\dots,h_m$ of $\mathsf F_n$. 
Then the set   $\{\phi_1,\dots,\phi_m\}$ is said to be
  {\it logically consistent}  if some  truth-valuation 
  assigns value 1 to each formula  $\phi_i.$
In equivalent algebraic terms, 
  $\eta(h_1)=\dots=\eta(h_m)=1$ for some 
  $\eta\in \hom(\mathsf F_n)$.
  
  \index{logically consistent set of formulas}
  

  \begin{corollary}
  \label{corollary:justification} 
\,\,\,  Let  $E=\{h_1,\dots,h_m\}$ be a finite subset
of the boolean algebra $\mathsf F_n$. Let  the function
 $\beta\colon E\to \yesno$  assign the constant
 value $1$ to each element of  $E$.
  Then the following conditions  are equivalent:
  \begin{itemize}
  \item[(i)]  
 $E$ can be coded by a  logically consistent set
 $E'$  of boolean  formulas.
 
 \smallskip
 \item[(ii)] $\,$  $\beta$  is (de Finetti) consistent in the sense of
Definition \ref{definition:consistency}.

 \smallskip
  \item[(iii)]   $\,\,\,\beta$ can be extended to
  a homomorphism  of $\mathsf F_n$ into $\yesno$. 
  
   \smallskip
  \item[(iv)]   $\,\,\,\beta$ can be extended to
  a state of $\mathsf F_n$.
  \end{itemize} 
\end{corollary}

\begin{proof}
 (ii)$\Leftrightarrow$(iv)\,  By  Theorem 
\ref{theorem:dutch}.

\smallskip 
 (i)$\Rightarrow$(iii)\, By hypothesis, some  truth-valuation $v$
 assigns   1 to all formulas in  $E'.$ 
 Equivalently,   $\eta(h_1)=\dots=\eta(h_m)=1$ for some 
  $\eta\in \hom(\mathsf F_n)$.
 Thus,  $\eta$ extends $\beta.$
 
 \smallskip
(iii)$\Rightarrow$(ii)  Every homomorphism of $\mathsf F_n$ into $\yesno$
is a state. Now apply   Theorem 
\ref{theorem:dutch}.

\smallskip
(ii)$\Rightarrow$(i)  Assume (i) does not hold, with the intent
of proving  that $\beta$  is inconsistent in the sense
of    Definition \ref{definition:consistency}. 
 To this purpose,
for each $j=1,\dots,m$,\,\,\,
Betty places the  stake  $s(h_j)=-1$.
As a result, she  now receives $m$   from Bookie.
 In any possible
 world  $\eta\in \hom(\mathsf F_n)$
she will have to return 1  to Bookie for each event
to which $\eta$ assigns the value 1.
  However, at least one event
$h=h_\eta\in E$  will be evaluated
0 by $\eta$, because
 the assumed logical inconsistency  of  $E'$
 entails that no  truth-valuation $v$ assigns 
value 1 to all formulas of  $E'.$ 
 Thus  Betty has a net  profit  $\geq 1$
  in any possible world  $\eta$. 
By Definition
 \ref{definition:consistency}, $\beta$ is inconsistent.
\hfill{$\Box$}
\end{proof}

  Using the fact that
 every boolean algebra $A$ is the homomorphic
 image of some free boolean algebra,
 with some extra work 
  Corollary \ref{corollary:justification}
  can be shown to hold for $A$. 
  

\subsection{De Finetti's Consistency Theorem for all boolean algebras}
\label{section:infinite}
\index{infinite boolean algebra}

Generalizing Theorem \ref{theorem:dutch},  in
Theorem  \ref{theorem:dutch-general}  we  will
 show  that our restriction
to finite boolean algebras is inessential. 
So let us relax this restriction: 
$$
\fbox{
 \mbox{ $A$\,\, henceforth  
stands for an arbitrary (finite or infinite) boolean algebra.  }
}
$$

 \index{$\hom(A),$ the homomorphisms of $A$  into  $\yesno$} 
 \index{consistent!book}

\medskip
\noindent
The set $\hom(A)$ of  homomorphisms
of $A$ into the two-element boolean algebra  $\yesno$ 
is no longer indexed by the atoms of $A$.
Suffice to say that  an infinite boolean
algebra $A$ need not have atoms.
Thus the definition of consistency
gets  the following  form, which is equivalent to
Definition   \ref{definition:consistency}  for finite boolean
algebras, and encompasses the general case 
when the boolean algebra $A$ is infinite:

\index{inconsistent subset of a boolean algebra}
\index{consistent!subset of a boolean algebra}  
\index{inconsistent book}
  
\begin{definition}
\label{definition:consistency-infinite-ba}
{\rm Let  $E =\{h_1,\dots h_m\} $ be a subset of
  a boolean algebra $A$ and 
 $\beta \colon E \to  \interval$
 a function.  Then $\beta$ is said to be
  {\em inconsistent in $A$}
if  there is a function  $ s\colon  E \to \mathbb R$ such that 
 \begin{equation}
\label{equation:inconsistency-general}
\sum_{j=1}^m s(h_j)(\beta(h_j)-\eta(h_j)) < 0
\,\,\mbox{ for every }\eta\in \hom(A).
\end{equation}
If $\beta$ is not inconsistent  in $A$
 we say it is {\em consistent in $A$}.
 }
\end{definition}
The definition of a {\it state} of $A$ is verbatim  the
same as Definition \ref{definition:state}.

\index{state!of a boolean algebra}


\begin{lemma}
\label{lemma:inserimento}
Let  $B$ be a {\em finite} boolean algebra and
$B'$ a subalgebra of $B$.
 
\smallskip
\noindent
(i) The homomorphisms of 
$B'$ into $\yesno$
are precisely the restrictions to $B'$ of the homomorphisms
of $B$ into $\yesno.$

\smallskip
\noindent
(ii) The states of $B'$ are precisely the restrictions to $B'$
of the states of $B$. 
\end{lemma}

\begin{proof}
(i)
Let 
 $\epsilon \in \hom(B)$. Trivially,  the restriction
 $\epsilon \restrict B'$ of $\epsilon$ to $B'$
 is a homomorphism of $B'$ into $\yesno$.

Conversely, let  $\eta\in \hom(B'),$
with the intent of
extending $\eta$ to a some
$\eta^*\in \hom(B)$. 
By  Corollary
\ref{corollary:tutto}(i) we may write   $\eta=\eta_e$
for a unique atom $e$ of $B'$.
By Proposition \ref{proposition:sum-of-atoms},
each atom $a$ of $B'$ is a nonzero
element of $B$, and dominates a nonempty set
$[a]$ of atoms of $B$.
  If  $a \not= b\in \at(B')$  then $[a]\cap [b]=\emptyset.$ 
  (For  otherwise, some atom $c$ of $B$  is dominated by both
$a$ and $b$, whence  $c$  is dominated by $a\wedge b=0$,
which is impossible.) 
 As a consequence, upon writing  
 $$
c\thickapprox d  \mbox{\,\,\,\,iff $c$ 
and $d$ are dominated by the same atom
$a$ of $B'$, $(c,d\in\at(B))$,}
 $$
we obtain an equivalence relation $\thickapprox$ 
over $\at(B)$.
The  function
 $a\mapsto [a], \,\,a\in \at(B')$  maps 
 $\at(B')$ onto the set of $\thickapprox$-equivalence
 classes.
Furthermore, for  each  $x\in B'$ 
letting $a_1,\dots,a_m$ be the atoms of $B'$ dominated
by $x$, we have the identity
\begin{equation}
\label{equation:vee}
x=  \bigvee\{b\in\at(B)  \mid b\in  ([a_1] \vee\dots\vee[a_m]) \}.
\end{equation}
Turning to our homomorphism $\eta_e\in \hom(B')$,
let us arbitrarily pick an atom  $e^*\in [e]$,  and let 
$\eta_{e^*}\in \hom(B)$ be the corresponding
homomorphism.  By a final application of  Corollary
\ref{corollary:tutto},  for every $y\in B$,\,\,\,
$\eta_{e^*}(y)=1$ iff $y$ dominates $e^*.$
An easy verification using \eqref{equation:vee} shows that
the homomorphism  $\eta^*=\eta_{e^*}$ is
an  extension of $\eta_e.$

\medskip
(ii)   The restriction to $B'$ of
any state of $B$ 
 is a state of $B'$.  
 Conversely, let $\tau$ be a state of $B'.$
 With the notation of (i),
  for each $a\in \at(B')$ arbitrarily pick 
 an atom $a^*\in [a]\subseteq \at(B)$.
 Let  the function $\tau^\circ \colon \at(B)\to \yesno$ be defined by
 $$
 \mbox{$\tau^\circ(a^*)=\tau(a)$ for each $a\in \at(B')$, and
  $\tau^\circ(b)=0$ for all other atoms of $B.$
  }
  $$ 
By  Corollary  \ref{corollary:vi-sigma}(iii),
 any state of the finite boolean algebra
$B$ is uniquely determined by the values it
gives to the atoms of $B$.
It is now easy to verify  that    $\tau^\circ$ uniquely determines
a  state $\tau^*$ of $B$ which extends $\tau.$
   \hfill{$\Box$}
 \end{proof}
 
 \medskip
 
 \begin{lemma}
 \label{lemma:missing} Let $B'$ be a finite
 subalgebra of an {\em infinite}  boolean algebra
 $A$. 
 
 \smallskip
(i)  
 The states of $B'$ are precisely the restrictions to $B'$
of the states of $A$. 

 \smallskip
(ii) 
The homomorphisms of 
$B'$ into $\yesno$
are precisely the restrictions to $B'$ of the homomorphisms
of $A$ into $\yesno.$
 \end{lemma}
 
 \begin{proof}
 (i) Let $\sigma$ be a state of $B'.$
 By   Lemma \ref{lemma:inserimento}(ii),  for  every
 {\it finite}  subalgebra $B$ of $A$ containing $B'$,
 $\sigma$ is extendable to a state of $B$. 
Let the  set  $\tilde{S}_{\sigma,B}\subseteq \interval^A$
 be defined by
 $$
 \tilde{S}_{\sigma,B}=\{f\colon A\to \interval
\, \mid\,  f\restrict B \mbox{ is a state of $B$ extending } \sigma\}.
 $$
By  definition
 of  the product topology of $ \interval^A$,
   $\tilde{S}_{\sigma,B}$ is a nonempty   {\it closed}
 subset of the  compact Hausdorff space $\interval^A$.

  For any
 finite family of finite subalgebras $B_1,\dots,B_u$
   of $A$ containing
 $B'$ the intersection  $I$ of the closed sets
 $S_{\sigma,B_1},\dots,S_{\sigma,B_u}$  is nonempty. 
 As a matter of fact, 
 letting $B_{u+1}$ be the subalgebra
  of $A$
 generated by the finite set $B_1\cup \dots \cup B_u$,
 it follows that $B_{u+1}$ is finite 
  (see  Section \ref{section:recap}), and the set
 $I$ contains the nonempty set $\tilde S_{\sigma, B_{u+1}}$. 
Next  let 
  $$
  I^*=\bigcap\{\tilde S_{\sigma,B}\mid B
  \mbox{ a finite subalgebra of $A$ containing $B'$}\}.
  $$
      From the compactness of 
  the Tychonoff cube  $\interval^A$
  it follows that 
$I^*$  is nonempty.
Any element of $I^*$ is a state of $A$ extending $\sigma.$

\medskip
(ii) A routine variant of the proof of (i)
using    Lemma \ref{lemma:inserimento}(i). 
 \hfill{$\Box$}
\end{proof}

\medskip
 
  Theorem \ref{theorem:dutch}  has the following
generalization to every (finite or infinite) boolean algebra $A$:

\index{theorem!de Finetti consistency theorem, general case}
\index{de Finetti!consistency theorem for any boolean algebra}

\index{consistency!=extendability to a state}

 \begin{theorem} 
 {\rm (De Finetti consistency theorem, general case)}
 \label{theorem:dutch-general}
Let $E$ be a finite subset of a 
  boolean algebra
$A$  and $\beta\colon E\to \interval$ a function.
Then  $\beta$ is consistent in $A$
  iff  $\beta$ is extendable  to a state of $A$.
\end{theorem}

\begin{proof} 
Say $E=\{h_1,\dots,h_m\}.$
The
subalgebra $B_E$ of $A$
 generated by $E$ is finite. (As noted in 
 Section \ref{section:recap},  the number
 of elements of $B_E$ is $\leq 2^m.$)
 
 \medskip
 {\it  $(\Rightarrow)$}
Let   $s\colon E\to \mathbb R$ be Betty's  bet
on the events in $E$. 
Let us agree to say that
any    $\eta=\eta_s\in \hom(B_E)$ such that
\begin{equation}
\label{equation:witness}
\sum_{l=1}^m s(h_l)(\beta(h_l)-\eta_s(h_l)) \geq 0
\end{equation}
is a {\it   witness 
 of  the $s$-consistency of $\beta$ in $B_E$.}  
  Since, by assumption,  $\beta$
 is consistent in $A$, for every  bet
 $t \colon E\to \mathbb R$  there is a witness
 $\theta=\theta_t\in \hom(A)$   of the $t$-consistency of
 $\beta$ in $A$.\footnote{The definition of 
 $\theta_t\in \hom(A)$ witnessing
 the $t$-consistency of $\beta$ {\it in  $A$} 
 is precisely  the same as 
 \eqref{equation:witness}, with $t$ in place of $s.$}
A fortiori,   the restriction
 of  $\theta_t$ to $B_E$ witnesses the $t$-consistency
 of $\beta$ in $B_E$. 
 Since $t$ is any arbitrary bet on 
 $E$,  $\beta$ is consistent in $B_E.$
  By Theorem \ref{theorem:dutch},
  $\beta$ is extendable to a state $\tau$ of $B_E.$
  By    Lemma \ref{lemma:missing}(i), 
$\tau$ is extendable to
 a state $\rho$ of $A.$ A fortiori,  $\beta\subseteq \tau$
is extendable to  $\rho$.

  \index{product!topology}

  \index{Tychonoff cube} 

\medskip
 
{\it  $(\Leftarrow)$} 
Let  $\sigma$ be a state of $A$ extending $\beta$. The
 restriction  $\sigma\restrict
 B_E$ is  a state of $B_E$ extending $\beta.$ 
 By Theorem \ref{theorem:dutch},\,\,
 $\beta$ is consistent in $B_E.$
 We have to prove that $\beta$ is consistent  in $A$.
Arguing by way of contradiction,  let us assume 
Betty can devise a bet   $s\colon E \to\interval$ such that 
$$
\sum_{j=1}^m s(h_j)(\beta(h_j)-\eta(h_j)) < 0
\,\,\mbox{ for every }\eta\in \hom(A).
$$
By Lemma  \ref{lemma:missing}(ii),\,\,\,
$
\sum_{j=1}^m s(h_j)(\beta(h_j)-\epsilon(h_j)) < 0
\,\,\mbox{ for every }\epsilon \in \hom(B_E).
$
By Definition  \ref{definition:consistency-infinite-ba},
 $\beta$ is inconsistent  in $B_E$,  a contradiction.   \hfill{$\Box$}
\end{proof}

\begin{corollary}
\label{corollary:finitely-consistent}
{\it  Let $A$ be 
a boolean algebra. Then  $A$ has a state.
A function   $f\colon A \to \interval$  is 
  a state of $A$ iff $f$ is {\em finitely consistent},
in the sense that every finite restriction of $f$ is
consistent in $A$.}
\end{corollary}
\begin{remarks}
\label{remarks:ac}
{\rm
 \,\,\,(i) The proof of  
 Lemma  \ref{lemma:missing} 
rests on  
the compactness of  the Tychonoff cube $\interval^A$. 
In  ZF (Zermelo-Fraenkel)
 set-theory  this
is an equivalent reformulation of  the Axiom of Choice.
The dependence of de Finetti's consistency theorem
 \ref{theorem:dutch-general}  on this axiom is rarely
 made explicit  in the vast  literature on this theorem.

 \smallskip 
  (ii) From Lemmas
  \ref{lemma:inserimento}-\ref{lemma:missing}  
  and Theorem \ref{theorem:dutch-general}
  it follows  that a function 
    $\beta\colon E\to\interval$ is consistent in
a boolean  algebra $A\supseteq E$
iff it is consistent
in the algebra  $B_E$ generated by $E$ in $A$,
 iff it is consistent in any algebra
$A^*$ containing  $A$ as a subalgebra.
Hence the consistency of $\beta$ is largely indifferent 
to the chosen boolean algebra $A\supseteq E$.
However, the specification
of {\it some} ambient boolean algebra  $A$ for the
set $E$ of  ``events'' and their operations
considered in de Finetti's
consistency theorem is necessary to give
a meaning to   expressions such as ``in any possible
case'', or ``in any possible world'', which occur in the
definition of consistency.

 \smallskip 
(iii)  De Finetti's consistency  theorem
has two directions. He proved the
 ($\Rightarrow$)-direction   
  in  \cite[pp.309-312]{def-fundmath},
  and the $(\Leftarrow)$-direction
 in  \cite[p.313]{def-fundmath}.
Nevertheless,  purported 
``converses'' of de Finetti consistency
 ($\sim$ Dutch book)  theorem
exist  in the literature. 
}
  \end{remarks}
  %
%

\index{consistency!does not depend on the ambient algebra}
\index{Tychonoff cube} 
\index{Axiom of Choice}
\index{ZF set theory}
\index{compactness of the Tychonoff cube = Axiom of Choice}
\index{finitely consistent}
\index{state!=finitely consistent function}
%
\subsection{A self-contained proof of Gordan's Theorem}
\label{section:gordan}

\index{Gordan's theorem}
\index{theorem!Gordan}

In this section we give a self-contained proof
of Gordan's theorem, a basic ingredient of the proof
of de Finetti's consistency theorem.

%
%
%
%

\begin{theorem}
{\rm  (Gordan's theorem)}
\label{theorem:gordan}
Let $M$ be a real matrix with $n$ rows and $m$ columns. 
Then precisely one of the following conditions holds:
\begin{itemize}
\item[(i)]   There is a  (column) vector $s \in \mathbb R^m$ such that 
every  coordinate of the   vector $ M s \in \mathbb R^n$  is $< 0$.

\smallskip
\item[(ii)]     $\/$ There is a nonzero vector
$u \in \mathbb R^n$ such that every coordinate of
$u$ is   $\geq 0$ and $u^T M=0 = \mbox{the origin in}\,\,\,
\mathbb R^m.$\,\,\,(As above,
$u^T$ is the transpose of $u$.)
\end{itemize}
\end{theorem}

\begin{proof}
{\rm Conditions  (i) and (ii) in the statement of
Gordan's theorem 
\ref{theorem:gordan}
 are  incompatible.
For, if both are assumed to hold we have
the contradiction  
\begin{equation}
\label{equation:aut-aut}
0\,\,\,\,\,\,\, \not=\,\,
 \underbrace{{\,\,\,\,\,\,}u^T{\,\,\,\,\,\,}}_{\geq 0\dots\geq 0,\,\,\, u\not=0} 
\,\,\,\, 
   \underbrace{ (Ms)}_{<0\dots<0}\,\,\,
=\,\,\, (u^T M)\,  s\,\,\,=\,\,\,
0.
\end{equation}
There remains to be proved that if 
  (i) fails then (ii) holds.
  Failure of (i) means that
the range $R$  of the linear operator $M\colon \mathbb R^m\to \mathbb R^n$ is disjoint from the  south-west  open octant 
$O$ given by
$$
O= \{v\in \mathbb R^n \mid \mbox{ each coordinate of }
v \mbox{ is } <0 \}.
$$
$R$ is a linear subspace
 of $\mathbb R^n$ containing the origin.  
 The hyperplane separation 
lemma  \ref{lemma:separation}
\footnote{A self-contained proof is given  below.} 
then  yields a nonzero
column  vector $u \in \mathbb R^n$, along  with 
a   hyperplane $H$  \,\,such that 
$$
R\subseteq H= \{x\in \mathbb R^n\mid  u^T  x = 0 \}
\,\,\,\mbox{ and }\,\,\,\,
O\subseteq \{x	\in \mathbb R^n\mid   u^T  x < 0 \}.
$$
The latter  inclusion implies that all 
coordinates of $u$ are $\geq 0.$  
It is now easy to see that $u^TM =0$.
For otherwise,  (absurdum hypothesis),
 there is a vector  $q \in \mathbb R^m$ with
$(u^T M) q\not= 0.$ 
As a member of
$R$,  the vector  $Mq$  is orthogonal to
$u^T,$  whence $u^T(Mq)=0,$ a contradiction.  
We have thus proved  that
if  (i) fails then (ii) holds.}  
\hfill$\Box$
\end{proof}

\index{south-west open octant} 
\index{hyperplane separation}

%
 
 \index{closure of a convex set}
 
 To complete the proof of Gordan's theorem, as well
as of de Finetti's  theorems \ref{theorem:dutch} and
 \ref{theorem:dutch-general},    in 
this  section a routine proof is given of the hyperplane separation lemma  \ref{lemma:separation}. 
\footnote{Some proofs of de Finetti's consistency theorem in the literature  cite a ``hyperplane separation theorem" without 
indicating the specific
  result  needed  for the proof.}

\begin{fact} 
\label{fact:closure}
The closure  $\cl X$ of a convex set 
$X\subseteq \mathbb R^n$ is convex.
\end{fact}

\begin{proof} If $X=\emptyset$ we are done.
If $X\not=\emptyset$, for any  $x,y\in \cl X$
we must prove that the segment $[x,y]=\{(1-\lambda)x+\lambda y
\mid \lambda \in \interval \}$ lies in $ \cl X$. 
By definition of closure,   $X$ contains converging 
sequences  $x_n\to x$ and $y_n\to y$. 
Since $X$ is convex,
the  interval
$[x_n,y_n]$ lies in $X$. The sequence of midpoints
$m_n$   of  $[x_n,y_n]$ 
converges
to the midpoint $m$ of  $[x,y]$. As a limit of 
a convergent sequence
of points of $X$,  $m$ lies in $\cl X$.  More generally,
for each $\lambda\in \interval$ the sequence
$(1-\lambda)x_n+\lambda y_n$ converges to 
$(1-\lambda)x+\lambda y$. So $(1-\lambda)x+\lambda y$
lies in $\cl X$.  As $\lambda$ ranges over
$\interval$, all points of  $[x,y]$ are obtained. They all
lie in $\cl X$. 
\hfill{$\Box$}
\end{proof}

 \index{convex!set}
 \index{$|v|$ the euclidean norm of vector $v$}
\index{$|w|$ the length  of vector $w$}

\medskip
Let us recall that the scalar product
of two (always column) vectors $u,v\in \mathbb R^n$
is given by  matrix multiplication $u^Tv$,
with $u^T$ the transpose of $u$.

For $w\in \mathbb R^n$ we let $|w|$
 denote its euclidean norm,  (or ``length'')  $(w^Tw)^{1/2}$.

\index{scalar product}
\index{matrix multiplication}
\index{transpose}
\index{shortest vector in a closed set}

\begin{fact}
\label{fact:molteplice}  
Let $X\subseteq \mathbb R^n$ be
a nonempty closed convex set.

\smallskip
(i)   $X$ has a unique {\em shortest}  vector, i.e., a
vector $x_*$ with   $|x_*|= \inf\{|x|\,\mid\, x\in X\}$.

\medskip
(ii) Let\,\,\,  $t \in \mathbb R^n.$
Then there is a unique point $u$ in  $X$ 
 such that $|u-t|=\inf\{|x-t|\,\,\mid \,\,x\in X\}$.

\medskip
(iii) Assume $0\notin X.$
Let $x_*$ be the shortest vector in $X$  given by (i).  
Then for all $y\in X$,\,\,
$x_*^Ty> |x_*|^2/2.$

\medskip
(iv)   
Let $x_*$ be the shortest vector in $X$.
Assume  $0\notin X.$ Then the hyperplane  
$$
H=\{z\in
 \mathbb R^n\mid x_*^T z =|x_*|^2/2\}
 $$
{\em strongly separates} the origin from
 $X$.\footnote{i.e.,
  the distance of the origin from $H$ is
  $>0$,    the distance from $H$
  of  every  $y\in X$ is $>0$, and 
   $X$ and the origin
 are  contained 
 in opposite closed half-spaces with boundary  $H$.} 
\,\,\, More generally, for any 
   nonempty closed convex set   $Y\subseteq \mathbb R^n$,  
   point $e\in \mathbb R^n\setminus Y$,
 letting 
$x_*$ be  the unique point of $Y$ closest to $e$, we have 
$
 (x_*-e)^T (x-e) > |x_*-e|^2/2 \mbox{ for each }x\in Y.
$
In other words,
 the hyperplane 
 $$
 K=\{z\in
 \mathbb R^n\mid  (x_*-e)^T (z -e) = |x_*-e|^2/2\}
 $$
  strongly separates $e$ from $K$.
\end{fact}

\index{strong separation}

\begin{proof}
(i)  If the origin $0$ belongs to $X$ we have nothing to prove.
Otherwise, let   $\xi$  be  shorthand  
for $ \inf\{|x|\mid x\in X\}$.
Since $X$ is nonempty,  it is no loss of generality
 to assume that
$X$ is bounded. By definition of infimum
  there is a sequence  $x_n\in X$ with
$|x_n|-\xi <1/n.$ There is a convergent subsequence  $x_{n_i}$.
The  point 
$x_*=\lim x_{n_i}$  belongs to $X$, because $X$ is closed.
Since $x\mapsto |x|$ is a continuous real-valued function,
then   
 $\xi =\lim |x_{n_i}| =|\lim x_{n_i}|=|x_*|.$ To prove
the uniqueness of $x_*$ suppose (absurdum hypothesis)
 $y\in X$ is different from 
$x_*$ and  has the same length
as $x_*$. The triangle with vertices $0,x_*,y$ is isosceles.
The midpoint of  the interval $ [x_*,y]$  
 has a  distance $< |x_*|$ from 0, which is
impossible.  

\medskip
(ii)  The translated set $X-t=\{x-t\mid x\in X\}$
is also nonempty closed and convex. 
Distances  $|a-b|$  are
preserved under translation.  Now apply (i).

\medskip
(iii)
Arguing by way of  contradiction, let us assume
that there exists
 $y \in X$ such that 
\begin{equation}
\label{equation:serviva}
x_*^Ty \leq  |x_*|^2/2.
\end{equation}
 The closed interval
$[x_*,y]=\{(1-\lambda)x_*+\lambda y \mid \lambda\in [0,1]\}$  is
contained in $X$. 
For all  $0\leq \lambda\leq 1$ we then have:
\begin{eqnarray*}
|x_*|^2&\leq& |(1-\lambda)x_*+\lambda y|^2,
\,\,\,\mbox{ since $x^*$ is the shortest vector in $X$}\\
{}&\,\,=\,\,& (1-\lambda)^2|x_*|^2 +
\lambda^2|y|^2 +2(1-\lambda)\cdot \lambda\, x_*^Ty\\
{}&\,\,=\,\,&  |x_*|^2 +\lambda^2|x_*|^2
-2\lambda |x_*|^2 
+
\lambda^2|y|^2
+
(1-\lambda)\lambda|x_*|^2, \mbox{ by (\ref{equation:serviva})}.
\end{eqnarray*}
Therefore, 
\begin{eqnarray*}
0&\,\,\,\leq\,\,\, &
 \lambda^2|x_*|^2-2\lambda |x_*|^2 +\lambda^2|y|^2 +
 \lambda |x_*|^2-\lambda^2|x_*|^2\\ 
& \,\,\,=\,\,\, & -\lambda |x_*|^2+\lambda^2|y|^2.
\end{eqnarray*}
We then have  
 $|x_*|^2\leq \lambda |y|^2$ for each $\lambda >0.$
  The only 
possibility is  $|x_*|^2=0$ whence 
$x_*$ coincides with the origin of $\mathbb R^n$, a contradiction.

\medskip
(iv) Immediate from (iii).
\hfill{$\Box$}
\end{proof}

\index{strong hyperplane separation} 
\index{theorem!supporting hyperplane}
\index{supporting hyperplane}
\index{strong separation}

\begin{fact}
{\rm (Supporting hyperplane theorem)}
\label{fact:supporting}
%
%
 Let $X$ be a   closed
 convex subset of
$\mathbb R^n$ having a boundary
point \,\,$b$.
 Then there is a hyperplane $H$ containing $b$,
such that $X$ is contained in  one of the two closed
half-spaces bounded  by $H$.
$H$  is known as a {\em supporting hyperplane} for $X$
at $b$.
\end{fact}

\begin{proof}    By definition of boundary, $b$ lies in
the closed set  $X$, and
there is a sequence of points $b_n\notin X$ with
$b_n\to b.$  For each $b_n$, let $x_n\in X$ be the closest
point to $b_n$ as given by Fact 
\ref{fact:molteplice}(ii). Let  $u_n$ be the unit vector
$({x_n-b_n})/{|x_n-b_n|}.$ By Fact
\ref{fact:molteplice}(iv),
\begin{equation}
\label{equation:crossedcircle}
\frac{|x_n-b_n|}{2}< u_n^T(x-b_n) \,\,\,  \mbox{ for each $x\in X$,
in particular for $x=b$}.
\end{equation}
The bounded sequence  $u_n$ 
has a convergent subsequence, say
converging to the unit vector $u.$  Then by 
(\ref{equation:crossedcircle}), 
$$
\lim \frac{|x_n-b_n|}{2}\leq\lim u^T_n(b-b_n)=u^T\lim(b-b_n)=0, 
$$
whence
\begin{equation}
\label{equation:lorena}
\lim \frac{|x_n-b_n|}{2}=0.
\end{equation}
 From (\ref{equation:crossedcircle})-(\ref{equation:lorena}),
  for each  $x\in X$ we have
$$
0 \leq \lim\,  u_n^T(x-b_n) =
u^T(x-b)   \mbox{  whence } u^Tb\leq u^T x.
$$
Thus the hyperplane $H=\{x\in \mathbb R^n\mid u^Tx=u^Tb\}$
supports $X$ at $b$. 
\hfill{$\Box$}
\end{proof}

\index{octant}
\index{south-west open octant}
\index{hyperplane separation theorem}
\index{theorem!hyperplane separation}

We are now ready   to prove the specific
hyperplane separation lemma needed for the
proof of Gordan's theorem and, ultimately, for the proof
of de Finetti's consistency theorem.

\begin{lemma}
{\rm (Hyperplane Separation Theorem)}
\label{lemma:separation} 
 Let $O\subseteq \mathbb R^n$ be the  set  of
all vectors whose coordinates are $<0.$
We say that $O$ is 
the {\em south-west open octant}.
Let $R$ be a  linear  subspace
of $\mathbb R^n$
disjoint from $O$.
%
 Then for some vector
$u\in \mathbb R^n$  the hyperplane
$H=\{x\in \mathbb R^n\mid u^T x=0\}$
has the following (separation) property for 
 $O$ and $R$:
 $$
 \mbox{$R\subseteq H$, \,\,\,(i.e., $u^T r=0$ for all $r\in R$).}
 $$
 and
$$
\mbox{$u^T y < 0$ for all $y \in O$, whence 
$H\cap O \mbox{\,\,is empty}.$}
$$ 
\end{lemma}

\begin{proof} The set $R+O=\{r+o\in \mathbb R^n\mid r\in R,
\,\,o\in O \}$ is convex and does not contain the origin. 
Its  closure  $\cl(R+O)$ is convex, by Fact \ref{fact:closure}.
Since $R$ and $O$ are disjoint,  the origin is
a boundary point of $\cl(R+O)$. 

\smallskip 
Fact \ref{fact:supporting} provides
a  supporting hyperplane $H$ for  $\cl(R+O)$ at 0. 

\smallskip
 \noindent
 Since $R$ is linear and  $0\in R\cap H$ then  
 $R$ is contained in $H,$
 \begin{equation}
 \label{equation:inclusion}
 R\subseteq H.
 \end{equation}
By construction, 
 $O$ is contained  in one of the two
  {\it open} half-spaces with boundary  $H$.
The other open half-space  contains a 
 vector
  $u$ orthogonal to $H$ such that 
 $u^T y<0$ for all $y\in O$.
 For any such  $u$, by (\ref{equation:inclusion}),
  we automatically have
 $u^T r = 0$ for all $r\in R.$ 
  \hfill{$\Box$}
\end{proof}

  \index{open half-space} 

\index{theorem!of the alternative}

\medskip
The proof of  Gordan's theorem is now complete,
and so is the proof of 
 de Finetti's consistency theorem \ref{theorem:dutch},
as well as of its generalization  \ref{theorem:dutch-general}
to all boolean algebras.

\begin{remarks} 
{\rm
The main effect of de Finetti's consistency theorem
is summarized  by de Finetti himself in his quote
at the beginning of this paper:
all  the results of probability theory are nothing 
more than consequences of his definition of consistency. 
In particular, the traditional ``{\it axiom}  of additivity for the
probability of incompatible events''
is shown by Theorem
\ref{theorem:dutch}
 to be a {\it consequence of de Finetti's definition}
 \ref{definition:consistency} of  consistency.
In a nutshell:
\begin{equation}
\label{equation:uno}
  \fbox
  {
  \parbox{0.83\linewidth}
{
\center
{
consistency + incompatibility $\Rightarrow$
 \,\,additivity axiom
}\\[0.3cm]
}
}
\end{equation}

}
\end{remarks}

 \bigskip
\section{De Finetti's Exchangeability Theorem}
\label{chapter:exchangeability}

This chapter provides a  self-contained  proof of 
de Finetti's  exchangeability theorem, a seminal result
which he first proved
  in his 1930 paper
 \cite{def-1930} and then  in \cite{def-poincare}. 
We only use the language of  boolean algebras, 
doing without notions such as 
probability  space,   
  random variable,  expectation, 
   conditional, moment,    Radon measure,
martingale, variously present in the  
 literature on this theorem.
In a final section, the
  original formulation of de Finetti's theorem
will  be easily  recovered from our proof. 
The   proof given here, although elementary,  may discourage the reader unfamiliar with long combinatorial calculations. 
Since this chapter is independent of the rest of this 
paper, it can be skipped on
a   first reading.

\subsection{Product  states and exchangeable states}
\label{section:product-exchangeable}

\index{de Finetti!exchangeability theorem}
\index{de Finetti!theorem, by antonomasia}
\index{theorem!de Finetti exchangeability}
   \index{$S(A)$, the state space of $A$}
   \index{state!space $S(A)$ of $A$}
 \index{free!generating set}
 \index{free!boolean algebra  $\mathsf F_\omega$}
\index{$\mathsf F_\omega$, the free countable boolean algebra}

\medskip
Let $A$ be a (finite or infinite) boolean algebra.
  By definition
 of  product topology,   the  set 
$S(A)$ of states of $A$,  
equipped with the restriction of the  product topology of  
  $\interval^A$ is  
  a convex compact subspace  of    $\mathbb R^A$. 
Generalizing the definition of
$\mathsf F_n$ given  in Section  \ref{section:recap},
 let   
 $$
 \mbox{$\mathsf F_\omega$ be the
 {\it free boolean algebra over
the free generating set}  $\{X_1,X_2,\dots\}.$
}
$$
Equivalently,  
 $\{X_1,X_2,\dots\}$ generates $\mathsf F_\omega$,  
and for every $m=1,2,\dots$ and $m$-tuple  
$(\beta_{1},\dots\beta_m) \in \{0,1\}^m$, 
\begin{equation}
\label{equation:t}
 X_{1}^{\beta_{1}}\wedge \dots, \wedge\,
X_{m}^{\beta_{m}}\not=0.\,\,\footnote{As in Section
\ref{section:recap},\,\,\,  $X_{i}^{\beta_i}=
X_{i}$ if $\beta_i=1$, and
 $X_{i}^{\beta_i}=\neg X_{i}$ 
if  $\beta_i=0$.}
\end{equation}

\index{conjunct}
\index{miniterm}

\noindent
Any element  
$t\in \mathsf F_\omega$ of the form 
(\ref{equation:t})
  is said to be  a {\it miniterm}
 of $\mathsf F_\omega.$ Each $X_{i}^{\beta_{i}}$
is  called  a {\it conjunct} of   $t$.
Since, as we have seen,  the set 
$\{X_{1},\dots,X_{m}\}$ freely generates 
 the  free  boolean algebra  
$\mathsf F_m\subseteq \mathsf F_\omega$, it follows that 
$t$ is  also a {\it miniterm } of $\mathsf F_m$.   We  let
$$
\mbox{
$\possf(t)$\,\,\,\,\, (resp., $\negsf(t)$)}
$$
 denote the number of
non-negated (resp., the number of negated)
conjuncts of $t$. Thus, e.g.,  $\possf(t)=\sum_{i=1}^m \beta_i.$
\index{$\possf(t)$, the number of non-negated conjuncts of 
the miniterm $t$}
\index{$\negsf(t)$, the number of negated conjuncts of 
the miniterm $t$}  
 
 \smallskip
Arbitrarily fix $p\in \interval$. Let the function $\mathsf f_p$
assign to every miniterm $t\in \mathsf F_\omega$
 the value $p^{\possf(t)}
(1-p)^{\negsf(t)}$,
\begin{equation}
\mathsf f_p(t)=p^{\possf(t)} (1-p)^{\negsf(t)}.
\end{equation}
In the particular case 
$p=1$, we have  $\negsf(t)=0$. We then set 
$\mathsf f_1(t)=1$.
Likewise,   we set 
$\mathsf f_0(t)=1$.  
Since
$\{X_1,X_2,\dots\}$ 
freely generates $\mathsf F_\omega$, the
 function $\mathsf f_p$ is well defined.
Upon writing 
 every element  $a\in \mathsf F_\omega$ as
a disjunction of miniterms of some free algebra $\mathsf F_m,$
it follows that (the actual choice of $m$ is immaterial, and)
{\it   the
function $\mathsf f_p$ is extendable to a unique state}
$\pi_p$ of $\mathsf F_\omega$,
\begin{equation}
\label{equation:product-state}
\pi_p(a)=\mbox{ unique extension of 
$\mathsf f_p$,\,\,\, \,\,\,\,   $(p\in \interval$)}.
\end{equation} 
For instance, 
$$
\pi_p((X_1\wedge X_2\wedge X_3)
\vee (X_1\wedge X_2 \wedge \neg X_3))=
p^3(1-p)^0+p^2(1-p)=p^2=
\pi_p(X_1\wedge X_2),
$$
in agreement with the identity
$(X_1\wedge X_2\wedge X_3)
\vee (X_1\wedge X_2 \wedge \neg X_3)
= X_1\wedge X_2$. 

\index{product!state}
  \index{state!product}

\medskip
\begin{definition}
\label{definition:product-state}
{\rm For every $p\in\interval$ we say that $\pi_p$ is 
a {\em product state}  of
 $\mathsf F_\omega$.  
The restriction $\pi_p\restrict \mathsf F_m$ of  $\pi_p$ to
  $\mathsf F_m$ is said to be a
  {\em product state}  of $\mathsf F_m.$
 
 \smallskip
An {\em exchangeable state} 
 of $\mathsf F_\omega$
is a state  $\sigma\colon \mathsf F_\omega\to \interval$ such that 
 for every miniterm $t$ the value 
$
\sigma(t)\mbox{ only depends on the pair of integers }
(\possf(t),\negsf(t)).
$
}
\end{definition}

\index{exchangeable state}
\index{state!exchangeable}
\index{product states are interchangeable}

\medskip 
The proof of the following proposition is immediate:

\begin{proposition}
\label{proposition:easy}

\smallskip
(i)  For any  $p\in \interval$  the   product
 state $\pi_p$ is an exchangeable 
 state of $\mathsf F_\omega.$ 
 
 \smallskip
 (ii) Every
 convex combination of product states
 of $\mathsf F_\omega$ in the vector space 
 $\mathbb R^{\mathsf F_\omega}\supseteq \interval^{\mathsf F_\omega}$
 is exchangeable.
  \end{proposition} 
  
  \subsection{The Exchangeability Theorem}
  \label{section:exchangeability}

 \medskip 
 \noindent
 In his 1930 paper   \cite{def-1930} 
 de Finetti vastly extended  Proposition
 \ref{proposition:easy}(ii)
 with his characterization of exchangeable states.
In our  boolean algebraic language,
``de Finetti theorem'' by antonomasia is as follows:

\index{exchangeability theorem!boolean formulation}
\index{theorem!exchangeability, boolean formulation}

 \begin{theorem}
 \label{theorem:exchangeability}
 Let $\mathsf F_\omega$  be the free boolean
 algebra over the  free generating set $\{X_1,X_2,\dots\}$.
 Let    $\sigma$ be a state of $\mathsf F_\omega$.
 Then $\sigma$ is exchangeable iff it 
   lies in the closure of the
set of convex combinations  of product 
states of $\mathsf F_\omega$
in the vector space
$\mathbb R^{\mathsf F_\omega}$
 endowed with the restriction of the product topology.
 \end{theorem}

 \begin{proof} 
 Self-contained proofs 
of deep results -- the target of this paper --  can be
long and  challenging.
As with  the consistency theorem,
 the reader's patient  study will be rewarded
with   knowledge of another far-reaching 
  de Finetti theorem.
  
\smallskip
$(\Rightarrow)$-direction.
Arbitrarily fix  $n=1,2,\dots$.
 The restriction
 $\sigma_n=\sigma\restrict \mathsf F_n$ is an
  {\it exchangeable state of $\mathsf F_n,$}
in the sense that   for every miniterm $t$ of $\mathsf F_n$, the value
 $\sigma_n(t)$ only depends on the pair of integers $(\possf(t),\negsf(t))$.
 For all integers $N> n$ and 
$K=0,\dots,N$ let  $\xi_{N,K}$ be the   
state of the free boolean algebra
$\mathsf F_N$ assigning the value
 $1/{N\choose K}$ to each miniterm 
 $u$ of $\mathsf F_N$ with  $\possf(u)=K$,
 and assigning 0 to the remaining $2^N-{N\choose K}$ 
 miniterms of $\mathsf F_N$.

%
%
\index{extremal}

We can easily verify  that  
 $\xi_{N,K}$  is  {\it extremal}
   in the convex set of
   exchangeable
 states of $\mathsf F_N.$
In other words,   $\xi_{N,K}$  cannot be expressed as
 a nontrivial convex combination of two distinct
 exchangeable states of $\mathsf F_N$.
 Furthermore,   $\mathsf F_N$ has
 no other extremal exchangeable states beyond 
 $\xi_{N,0},\dots,\xi_{N,N}.$

Since the restriction $\sigma_N=
 \sigma\restrict \mathsf F_N$  is an
 exchangeable state of $\mathsf F_N$  there are
  real numbers
 $$\lambda_{N,0},\dots,\lambda_{N,N}\geq 0
\,\,\, \mbox{ with }\,\, \sum_{l=0}^n \lambda_l =1$$
 such that  $\sigma_N$ agrees
 over $\mathsf F_N$ with
 \index{convex!combination}
 the convex combination $\sum_{K=0}^N
  \lambda_{N,K} \cdot \xi_{N,K}$
in the finite-dimensional  vector space
$\mathbb R^{\mathsf F_N}.$ 
Since the restriction function  
$$\psi\in S(\mathsf F_N)\mapsto \psi\restrict
\mathsf F_n\in S(\mathsf F_n)$$
 is linear,  the state  
$\sigma_n=\sigma_N\restrict \mathsf F_n
=\sigma\restrict \mathsf F_n$ agrees
over $\mathsf F_n$
with the convex combination
$\sum_{K=0}^N
  \lambda_{N,K} \cdot \xi_{N,K}\restrict \mathsf F_n.$ 
    In particular, 
\begin{equation}
\label{equation:erre}
\sigma(r)=
\sum_{K=0}^N
  \lambda_{N,K} \cdot \xi_{N,K}(r),\,\,\,
    \mbox{ for each miniterm } r \in \mathsf F_n 
    \mbox{ and }N>n. 
\end{equation}
For some
$n$-tuple
of bits $(\beta_1,\dots,\beta_n)\in \{0,1\}^n$, let 
the miniterm $t\in \mathsf F_n$ be defined by 
$$
t=X_1^{\beta_1}\wedge\cdots\wedge
X_n^{\beta_n}.
$$  
Suppose the miniterm $w=
X_1^{\beta'_1}\wedge\cdots\wedge
X_N^{\beta'_N}\in \mathsf F_N$ satisfies 
$t\geq w$. 
Since the set  $\{X_1,X_2,\dots,X_N\}$ freely generates
 $\mathsf F_N$ then $\beta_1=\beta'_1,\dots,\beta_n=\beta'_n$, whence
 in particular
 $$\possf(t)\leq \possf(w)\leq N-\negsf(t)=N-n+\possf(t).$$ 
  For each $K=\possf(t),\,\dots,N-n+\possf(t)$, we have
  $$
  \mbox{
 $v\leq t$ for precisely
 $ {N-n \choose K-\possf(t)}$  
   miniterms $v$ of $\mathsf F_N$ with $\possf(v)=K$.
   }
   $$ 
   As we already know,
   for any such 
   $v$,\,\,\, 
 $\xi_{N,K}(v)$ coincides with $1/{N\choose K}$. 
Since $t$ equals the disjunction 
of the  miniterms $u$ of $\mathsf F_N$ 
satisfying  $t\geq u$,
then 
$$
\xi_{N,K}(t)=
 \left\{
 \begin{array}{cl}
 \frac{{N-n \choose K-\possf(t)}}{{N\choose K}} 
& \,\,\,  \mbox{if}\,\, K = \possf(t),\dots,N-n+\possf(t) \\[0.3cm]
 0\,\,\,  &\,\,\,  \mbox{if}\,\,
    K =0, \dots,   \possf(t)-1, \,\,N-n+\possf(t)+1,  \dots,  N. 
  \end{array}
  \right.
 $$
Thus by (\ref{equation:erre}), 
 for any miniterm $t$ of $ \mathsf F_n$ 
and  $N>n$  we can write
\begin{equation}
\label{equation:sigma(t)}
 \sigma(t)=
  \sum_{K=\possf(t)}^{N-n+\possf(t)}
  \lambda_{N,K} \cdot 
  \frac {{N-n \choose K-\possf(t)}  }{{N\choose K}}\,.
\end{equation}

 \bigskip
 \noindent{\it Claim 1:}  
 For  every $n=1,2,\dots$
  and $\epsilon >0$  there is $N>n$
 such that for every miniterm $t\in \mathsf F_n$,\,\,  
 letting 
$$
\mbox{$k$ be shorthand for $\possf(t) $},
$$  
we have
$$
 \left| \sigma(t)-
 \sum_{K=0}^{N}\lambda_{N,K}\,\pi_{K/N}(t)\right|
 =
  \left| \sigma(t)-
  \sum_{K=0}^{N}\lambda_{N,K}(K/N)^k(1-K/N)^{n-k}\right|
   <\epsilon.
$$

By  way of contradiction, suppose
there are   $n$ and 
 $\epsilon>0$ such that for every  $N>n$ there is a miniterm 
$u \in \mathsf F_n$ satisfying
$
 \left| \sigma(u)-
 \sum_{K=0}^N \lambda_{N,K}\pi_{K/N}(u)\right|\geq\epsilon.
 $
 Since  $\mathsf F_n $ is finite there is 
 $\epsilon>0$ and $n$,   together with a miniterm $t\in \mathsf F_n$
 such that for infinitely many  $N$,\,\,\,  
$ \left| \sigma(t)-\sum_{K=0}^{N}\lambda_{N,K}\pi_{K/N}(t)\right|
\geq\epsilon.$
 %
%

 \smallskip
\noindent
{\it Case 1:} $\possf(t) \notin\{0,n\}$. 
 By (\ref{equation:sigma(t)}) we can write
\begin{eqnarray*}
 \sigma(t) &=&
  \sum_{K=k}^{N-n+k}
  \lambda_{N,K}  \cdot 
  \frac {{N-n \choose K-k}  }{{N\choose K}}=
  \sum_{K=k}^N
  \lambda_{N,K} \cdot  
  \frac{(N-n)!}{(K-k)!(N-n-(K-k))!} \cdot \frac{K! (N-K)!}{N!}
\\[2mm]
&=&
  \sum_{K=k}^{N-n+k}
  \lambda_{N,K}   \cdot
\frac{(K-k+1) \cdots K}{(N-n+1)\cdots N}
\cdot{(N-n-(K-k)+1)\cdots (N-K)}
\\[2mm]
 &=& 
  \sum_{K=k}^{N-n+k}
  \lambda_{N,K} \cdot 
 \frac{(K-k+1)\cdots K}{(N-n+1)\cdots (N-n+k)}
 \cdot
 \frac{(N-K-n+k+1)\cdots (N-K)}{(N-n+k+1)\cdots N}
\\[3mm]
 &=&
  \sum_{K=k}^{N-n+k}
  \lambda_{N,K} \cdot 
  \prod_{i=1}^k \frac{K/N-(k-i)/N}{1-(n-i)/N}\,\,\,\cdot\,\,\,
    \prod_{j=1}^{n-k} \frac{1-K/N-(n-k-j)/N}{1-(n-k-j)/N}.
\end{eqnarray*}

\noindent
Since  $k=\possf(t) \notin\{0,n\}$ and 
$\lambda_{N,K}\leq 1$,  
%
then
\begin{eqnarray*}
0 &=&\lim_{N\to \infty}\,\,\sum_{K=0}^k
 \lambda_{N,K}({K}/{N})^k
\left(1-{K}/{N}\right)^{n-k}\\
{}&=&\lim_{N\to \infty}\sum_{K=N-n+k}^N
 \lambda_{N,K}\left(K/N\right)^k
\left(1-{K}/{N}\right)^{n-k}.
\end{eqnarray*}   
%
%
%
%
So our absurdum hypothesis equivalently states that  
there is
 $\epsilon>0$ together with a miniterm $t\in \mathsf F_n$
 such that
 $$
 \left| \sigma(t)-\sum_{K=k}^{N-n+k}
 \lambda_{N,K}\pi_{K/N}(t)\right|\geq\epsilon
$$
 for infinitely many  $N$.
For 
$ i=1,\dots,k$ and  $j=1,\dots, n-k)$ let the rationals $c_i$
and $d_j $ be defined by
$$
c_i=\left(1-\frac{k-i}{K}\right)/\left(1-\frac{n-i}{N}\right)
$$
and
$$
d_j=\left(1-\frac{n-k-j}{N-K}\right)/
\left(1-\frac{n-k-j}{N}\right).
$$
Then 
\begin{eqnarray*}
\epsilon 
&\leq&
  \sum_{K=k}^{N-n+k}
 \lambda_{N,K}
 \left|
   \prod_{i=1}^k \frac{K/N-(k-i)/N}{1-(n-i)/N}\cdot
    \prod_{j=1}^{n-k} \frac{1-K/N-(n-k-j)/N}{1-(n-k-j)/N}
    -  \pi_{K/N}(t) \right |\\
&\leq&
  \sum_{K=k}^{N-n+k}
  \lambda_{N,K}\cdot
 \left|
   \prod_{i=1}^k c_i \cdot {K/N}\cdot
    \prod_{j=1}^{n-k} d_j \cdot (1-K/N)
-
(K/N)^k(1-K/N)^{n-k} \right |\\
&=& \sum_{K=k}^{N-n+k}     \lambda_{N,K}\cdot 
(K/N)^k\cdot(1-K/N)^{n-k} \cdot
\left| 1- \prod_{i} c_i  \prod_{j} d_j\right |\\
&\leq&
\sum_{K=k}^{N-n+k}     \lambda_{N,K}\cdot 
(K/N)^k\cdot(1-K/N)^{n-k} \cdot
\left(1- \left(1-\frac{k}{K}\right)^k
\cdot\left(1-\frac{n-k}{N-K}\right)^{n-k}\right ),  
\end{eqnarray*}
because for all $i,j$ \,\,\, 
  $
  1-\frac{k}{K}\leq c_i\leq 1\,\,\,\mbox{  and  }\,\,\,
  1-\frac{n-k}{N-K}\leq d_j\leq 1.
  $  
%

 \medskip
Fix $\eta>0$. For all sufficiently large $N$ we have
$k< \lfloor \eta N\rfloor  < \lceil N-\eta N\rceil < N-n+k$, where
$\lfloor x \rfloor$  (resp., $\lceil x\rceil $) 
is the largest integer $\leq x$
(resp., the smallest integer $\geq x$). 
As a consequence,

\begin{eqnarray*}
\epsilon\,\, 
&\leq&\,\,
\sum_{K=k}^{\lfloor \eta N\rfloor}    \lambda_{N,K}\cdot 
(K/N)^k\cdot(1-K/N)^{n-k} \cdot
\left(1- \left(1-\frac{k}{K}\right)^k
\cdot\left(1-\frac{n-k}{N-K}\right)^{n-k}\right)
\\
&+&
\sum_{K=\lfloor \eta N\rfloor +1}^{\lceil N-\eta N\rceil}   
 \lambda_{N,K}\cdot 
(K/N)^k\cdot(1-K/N)^{n-k} \cdot
\left( 1- \left( 1-\frac{k}{K}\right)^k
\cdot\left(1-\frac{n-k}{N-K}\right)^{n-k}\right)
\\
&+&
\sum_{K={\lceil N-\eta N\rceil} +1}^{N-n+k}    \lambda_{N,K}\cdot 
(K/N)^k\cdot(1-K/N)^{n-k} \cdot
\left( 1- \left(1-\frac{k}{K}\right)^k
\cdot\left(1-\frac{n-k}{N-K}\right)^{n-k}\right)
\end{eqnarray*}
\begin{eqnarray*}
&\leq& \,\,
\sum_{K=k}^{\lfloor \eta N\rfloor }    
\lambda_{N,K}\cdot 
\left(\frac{\eta N}{N}\right)^k\cdot 1\cdot
1\\
&+&  
\sum_{K=\lfloor \eta N\rfloor +1}^{\lceil N-\eta N\rceil}    
 \lambda_{N,K}\cdot 
1 \cdot 1  \cdot
\left( 1- \left(1-\frac{k}{ \eta N}\right)^k\cdot
\left(1-\frac{n-k}{N-(N-\eta N)}\right)^{n-k}\right)
\\
&+&  
\sum_{K=\lceil N-\eta N\rceil+1}^{N-n+k}   
 \lambda_{N,K}\cdot 
1 \cdot\left (1-\frac{N-\eta N+1}{N}\right)^{n-k} \cdot
\,\,1\\
&\leq&\,\,  
\sum_{K=k}^{\lfloor \eta N\rfloor }    
\lambda_{N,K}\cdot 
\left[   \left(\frac{\eta N}{N}\right)^k\right]
+ \sum_{K=\lfloor \eta N\rfloor +1}^{\lceil N-\eta N\rceil}   
   \lambda_{N,K}
 \cdot
\left[ 1- \left(1-\frac{k}{\eta N}\right)^k\cdot
\left(1-\frac{n-k}{\eta N}\right)^{n-k}\right] 
 \\
&+&
\sum_{K=\lceil N-\eta N\rceil+1}^{N-n+k}   
 \lambda_{N,K} \cdot
  \left[\left (\eta-\frac{1}{N}\right)^{n-k} \right].
\end{eqnarray*}

 \smallskip
\noindent
Since $\epsilon$ is fixed
we may choose  $\eta>0$ so small  that
for all sufficiently large $N$,

\begin{eqnarray*}
\epsilon &\leq& \sum_{K=k}^{N-n+k}     \lambda_{N,K}\cdot 
\left(
\left(\frac{\eta N}{N}\right)^k +
\left( 1- \left(1-\frac{k}{\eta N}\right)^k\cdot
\left(1-\frac{n-k}{\eta N}\right)^{n-k}\right) +
\left(\eta-\frac{1}{N}\right)^{n-k} 
\right)\\
&\leq&
 \sum_{K=k}^{N-n+k}    
  \lambda_{N,K}\cdot \frac{\epsilon}{1000}
 \leq  \frac{\epsilon}{1000},
 \mbox{ which is impossible.}
\end{eqnarray*}

\smallskip
\noindent
{\it Case 2:} $\possf(t) \in \{0,n\}$.  
Then a routine simplification of the
 proof of Case 1 again yields a contradiction.

\medskip 
Having thus settled our claim, from
$\mathsf F_1\subseteq \mathsf F_{2}\subseteq\cdots$ it follows  that 
for every $\epsilon>0$ and $n=1,2,\dots,$ there is
$N>n$ together with  a convex combination
of product states of $\mathsf F_N$
which agrees with
$\sigma$ over all miniterms of $\mathsf F_n$ up to  
 an error
$<\epsilon$.
Every   set
$\{a_1,\dots,a_z\}\subseteq \mathsf F_\omega$
 is contained in some
finitely generated free boolean algebra 
 $\mathsf F_n$, and hence
each  $a_i$ is a disjunction
of  miniterms of   $\mathsf F_{n}.$
In conclusion,
\begin{quote}
{\it For all $\{a_1,\dots,a_z\}\subseteq 
\mathsf F_\omega$ and
  $\epsilon >0$ there is
a convex combination
of product states of $\mathsf F_{\omega}$
agreeing with  $\sigma$ over 
$\{a_1,\dots,a_z\}$   up to an error $<\epsilon$.}
\end{quote}
By definition
of the product topology of the
compact space
$\interval^{\mathsf F_\omega}$,
\,\,this amounts to saying that 
$\sigma$ belongs
to the closure of the
set of convex combinations  of product states of
 $\mathsf F_\omega$
in the vector space  
$\mathbb R^{\mathsf F_\omega}$ 
equipped with the product topology.  

\medskip
$(\Leftarrow)$-direction. Easy now.
\hfill{$\Box$}
 \end{proof}
%
%
%
 
\subsection{De Finetti's  formulation of the Exchangeability Theorem}
\label{section:original}
Fix an exchangeable state $\sigma$
of the  boolean algebra  $\mathsf F_\omega$ freely generated by
$X_1,X_2,\dots.$
The  proof of Theorem \ref{theorem:exchangeability}  
yields
 integers   $0 < N_1<N_2<\dots,$ and for each $i=1,2,\dots$
  real numbers
 $\lambda_{N_i,0},\dots,\lambda_{N_i,N_i}\geq 0$
 summing up to 1,  along with  
 Borel probability measures
 \index{regular Borel probability measure}
 \index{probability!Borel measure}
  \index{probability!regular Borel measure}      
  $\mu_{i}$ on  $\interval$ 
such that
 for any miniterm $t\in \mathsf F_\omega$  
 \begin{eqnarray*}
  \sigma(t)&=&\lim_{i\to \infty}
  \sum_{K=0}^{N_i}\lambda_{N_i,K}\,(K/N_i)^{\possf(t)}
  \,(1-K/N_i)^{{\negsf(t)}}\\
&=&  \lim_{i\to \infty}
 \sum_{K=0}^{N_i}
  \lambda_{N_i,K} \,\pi_{K/N_i}(t)\\
  &=&\lim_{i\to \infty} \int_{\interval}
 p^{\possf(t)}\,(1-p)^{\negsf(t)}\,
   {\rm d} \mu_i(p).
  \end{eqnarray*}


\index{exchangeability theorem!original formulation}
  \index{regular Borel probability measure}
 \index{probability!Borel measure}
  \index{probability!regular Borel measure}

\smallskip
\noindent
Elementary measure theory  
will now yield  the original formulation \cite{def-1930} of
 de Finetti's 
exchangeability  theorem  by letting $P([0,1])$ be the 
 compact metric space
  of Borel probability measures on
 $\interval$ equipped with the weak topology.
  %
 %
The sequential compactness of  $P([0,1])$ yields  
  a subsequence of the  $\mu_{i}$
converging to some 
$\mu\in P(\interval)$.   Thus    
      for every miniterm  $t\in \mathsf F_\omega$
 $$ 
 \sigma(t) 
=  \int_{\interval}
 p^{\possf(t)}\,(1-p)^{\negsf(t)}\,{\rm d} \mu(p) 
 =  \int_{\interval} \mathsf f_p(t)\,{\rm d} \mu(p) =
  \int_{\interval} \pi_p(t)\,{\rm d} \mu(p).
    $$
Intuitively: 
{\it Any  exchangeable sequence 
of $\{yes,no\}$-events
$X_1,X_2,\dots$  is a ``mixture"\footnote{i.e., a
 sort of a   ``weighted average''
computed  by an integral.} of 
sequences of independent and identically
 distributed (i.i.d.) Bernoulli random variables.}
 
  \index{event}
  \index{de Finetti!event}

\index{mixture of i.i.d. Bernoulli random variables}
\index{independent!identically distributed 
Bernoulli random variables}
\index{Bernoulli random variable}

\begin{remarks}
{\rm 
De Finetti proved
the  {\it existence} of $\mu$  
  in the Appendix
   of \cite[pp.124-133]{def-1930}  using  characteristic functions, 
   the tools available  to him when 
he communicated this result  in  the 1928
International Congress of Mathematicians in Bologna. 
   In all his subsequent papers  on exchangeability
   he abandoned this formalism and used  
distribution functions, probability measures, 
 moments, and laws of large numbers.   
  In  \cite[Chapter 4, \S 31, p.121]{def-1930},
  the free generating set  $\{X_1,X_2,\dots, X_n \} $
  of $ \mathsf F_n$   is said to constitute  a 
   ``class of equivalent events''. 
   
   \index{class of equivalent events}
   \index{de Finetti!equivalent events}
      \index{equivalent events}
      
   A comprehensive advanced account of 
   the ramifications of exchangeability is given
   by  D.J. Aldous in   
``Exchangeability and related topics''.
 In: \'Ecole d'  \'Et\'e de Probabilit\'es de
Saint-Flour XIII--1983. Lecture  Notes in Mathematics, 
  Vol. 1117, Springer,
Berlin,  (1985), pp. 1--198. 
%
Further extensions of de Finetti's theorem to quantum 
states have been applied in other research
areas, including  quantum information theory.
%
 

}
 \end{remarks}

 \bibliographystyle{plain}

\begin{thebibliography}{11}
 
  \bibitem{Boole1854}
 G. Boole, 
  An Investigation of the Laws of Thought, on which are Founded the Mathematical
Theories of Logic and Probabilities,  
Walton and Maberley, London, 1854 (reprinted by Dover, NY, 1958).

 
 \bibitem{def-1930}
B. de Finetti, 
Funzione caratteristica di un fenomeno aleatorio.
Memorie della Reale
Accademia Nazionale dei Lincei.  Vol. IV, Fascicolo 5,
 (1930)  86-133.
 Reprinted in:
B. de Finetti, 
Opere Scelte. 
 Vol. 1. Unione Matematica Italiana, Edizioni Cremonese,  Florence,  2006,
 pp.109-157.
 
 



   \bibitem{def-fundmath}
 B. de Finetti,
Sul significato soggettivo della probabilit{\'a},
{Fundamenta Mathematicae,} 17 (1931) 298--329.
Translated into English 
as: On the Subjective Meaning of Probability. 
In:  P. Monari,  D. Cocchi (Eds.), 
{ Probabilit\'a e Induzione,} 
Clueb, Bologna, pp. 291-321, 1993.
Reprinted in \cite[pp. 191--222]{def-collected}.


      \bibitem{def-poincare}
B. de Finetti,
La pr{\'e}vision: ses lois logiques,
ses sources subjectives,
{Annales de l'Institut H. Poincar{\'e},} 7 (1937) 1--68.
%
 English translation by Henry E. Kyburg Jr.,
as:  Foresight: Its Logical Laws, its Subjective Sources.
In: H. E. Kyburg Jr.,   H. E. Smokler, 
 Studies in Subjective Probability,
 J. Wiley,  New York, pp. 93--158,
1964. Second edition published by Krieger,
New York, pp. 53--118, 1980.
Reprinted in \cite[pp. 335--400]{def-collected}.




  \bibitem{def-wiley1}
 B. de Finetti,
{Theory of Probability}.
Vol. 1,  John Wiley and Sons, Chichester, 1974.

  \bibitem{def-collected}
 B. de Finetti,
Opere Scelte. (Selected Works), 
 Vol. 1. Unione Matematica Italiana, Edizioni Cremonese,  Firenze,  2006. 
 




%
%
%
  
   
\end{thebibliography}


\end{document}